\newcommand{\vcorr}[3][1]{%
  \begingroup
    \tabcolsep=.5\tabcolsep
    \sbox0{%
      \begin{tabular}[b]{@{}l}%
        #3%
         \tabularnewline
      \end{tabular}%
    }%
    \settoheight{\dimen0 }{%
      \rotatebox{#2}{%
        \copy0 %
        \kern-\tabcolsep
      }%
    }%
    \rule{0pt}{#1\dimen0}%
    \setlength{\wd0 }{1em}%
    \setlength{\ht0 }{1em}%
    \rotatebox{#2}{\usebox{0}}%
  \endgroup
}

\documentclass[10 pt]{amsart}

\usepackage{amsmath}
\usepackage{amsfonts}
\usepackage{amsthm}
\usepackage{amssymb}
\usepackage{amscd}
\usepackage[all]{xy}
\usepackage{enumerate}
\usepackage{mathrsfs}
\usepackage{graphicx, bm, color, ulem}
\input xyimport.tex

\font\eu=eusm10 at 10pt

\def\eF{\text{\eu F}}
\def\eG{\text{\eu G}}

\title{The rationality of ineffective spin genus-$4$ thetanull loci}

\theoremstyle{plain}
\newtheorem{thm}{Theorem}[subsection]

\newtheorem{prop}[thm]{Proposition}

\newtheorem{cor}[thm]{Corollary}
\newtheorem{lem}[thm]{Lemma}
\newtheorem{cla}[thm]{Claim}

\theoremstyle{definition}
\newtheorem{defn}[thm]{Definition}

\newtheorem{nota}[thm]{Notation}

\newtheorem*{ackn}{Acknowledgement}

\theoremstyle{remark}
\newtheorem*{rem}{Remark}

\newcommand{\sD}{\mathcal{D}}

\newcommand{\sF}{\eF}
\newcommand{\sG}{\eG}
\newcommand{\sH}{\mathcal{H}}

\newcommand{\sK}{\mathcal{K}}

\newcommand{\sN}{\mathcal{N}}
\newcommand{\sM}{\mathcal{M}}
\newcommand{\sO}{\mathcal{O}}

\newcommand{\sS}{\mathcal{S}}
\newcommand{\sR}{\mathcal{R}}

\newcommand{\sT}{\mathcal{T}}
\newcommand{\sU}{\mathcal{U}}

\newcommand{\mC}{\mathbb{C}}

\newcommand{\mG}{\mathbb{G}}

\newcommand{\mP}{\mathbb{P}}

\newcommand{\Aut}{\mathrm{Aut}\,}
\newcommand{\Hilb}{\mathrm{Hilb}}

\newcommand{\PGL}{\mathrm{PGL}}

\newcommand{\VSP}{\mathrm{VSP}\,}

\numberwithin{equation}{section}
\author{Francesco Zucconi}

\address{D.I.M.I. \\
the University of Udine\\
Udine, 33100, Italy\newline
\texttt{Francesco.Zucconi@dimi.uniud.it}}

\begin{document}
\begin{abstract} In this paper, we show that 
 the divisor given by couples $[C,\theta]$ where $C$ is a curve of genus $4$ with a vanishing thetanull and $\theta$ is an ineffective thetacharacteristic 
is a rational variety. By our construction, it follows also that the analogous divisor in the Prym moduli space is rational.
\end{abstract}
\maketitle


\section{Introduction}

\subsection{Analog results in the previous literature}

The rationality problem for the coarse moduli space $M_g$ of smooth curves of genus $g$ is a classical problem posed by Francesco Severi; see: \cite{Se}. For $M_g$ and its Deligne-Mumford compactification ${\overline{M_g}}$, see: \cite{DM}.  The literature on this topic is vast; here we can also quote \cite{I}, \cite{HM},\cite{BK}, \cite{Sh3}, \cite{Kat}, \cite{ACG}.
For the rationality problem concerning geometrically defined subvarieties of $\overline M_g$ see: \cite{Bo}, \cite{Sh1}, \cite{Sh2}, \cite{CC}, \cite{Boh}.
 Furthermore there are others moduli spaces for curves.
 One of the most studied, also because of its importance in physics,
is the course moduli space ${{\rm{S^+_g}}}$ of even spin curves, see: cf. \cite{Bi}. It parameterises up to automorphisms  couples $(C,\theta)$ where $C$ is a smooth curve, $\theta$ is a divisor, $h^0(C,\sO_C(\theta))$ is an even number and $2\theta$ is linearly equivalent to the canonical divisor $K_C$; such a divisor $\theta$ is called an even thetacharacteristic.

\subsection{The divisor of curves with a vanishing theta-null}
 In \cite{cornalba} Cornalba constructed a compactification ${\overline{S^{+}_g}}$ which is compatible with ${\overline{M_g}}$. 
By \cite{Kat} we know that $\overline{{\rm{S^{+}_g}}}$ is rational where $g=2,3$. In \cite{TZ3} it is showed that $\overline{{\rm{S^{+}_4}}}$ is rational. Since the Kodaira dimension of $\overline{{\rm{S^{+}_g}}}$ is $\geq 0$ if $g\geq 8$, see: \cite{FaV}, the rationality problem for $\overline{{\rm{S^{+}_g}}}$ is open only for the case $g=5,6,7$.
Nevertheless, as in the classical case of $\sM_g$, there are some geometrically defined subvarieties of $\overline{{\rm{S^{+}_g}}}$ for which it is a natural problem to understand if they are rational or not.

A geometrically defined divisor which plays a fundamental role in the geometry of  $M_g$ is $\sM^{{\rm{null}}}_g$ which is the locus given by the classes of smooth curves having at least one even theta-characteristic such that  $h^0(C,\sO_C(\theta))\geq 2$. It is known that $\sM^{{\rm{null}}}_g$ is irreducible; see: \cite[Theorem 2.4]{Mo}. We recall that ${\overline{\rm{S^+_g}}}$ is irreducible, that there exists an open set of elements $[(C,\theta)]$ where $h^0(C,\sO_C(\theta))=0$ and that an even theta characteristic $\theta$ with  $h^{0}(C,\theta)> 0$ is said to be a vanishing thetanull. We consider the following two divisors:
$$\Theta_{g,\rm{null}}:=\{[C,\theta]\in {\rm{S^+_g}}\mid h^{0}(C,\theta)> 0\}$$
and
$$
{{{\rm{S}}^{{\rm{null}},0 }_{g}}}:=\{[C,\theta]\in {\rm{S^+_g}}\mid [C]\in\sM^{{\rm{null}}}_g,\,\,{\rm{and}}\,\,   h^{0}(C,\theta)=0\}.
$$

By the forgetful morphism ${\rm{S^+_g}}\to\sM_g$ the preimage ${{{\rm{S}}^{{\rm{null}}}_{g}}}$ of $\sM^{{\rm{null}}}_g$ is
\begin{equation}\label{dsunion}
{{{\rm{S}}^{{\rm{null}}}_{g}}}={{{\rm{S}}^{{\rm{null}},0 }_{4}}}\sqcup \Theta_{g,\rm{null}}\end{equation}
As their analogue in $M_g$ the divisors $\Theta_{g,\rm{null}}$ and ${{{\rm{S}}^{{\rm{null}},0 }_{g}}}$ are useful to describe the geometry of ${\overline{S^{+}_g}}$, see \cite[Theorem 1.1]{FaV}. Moreover genus $4$ curves with a vanishing thetanull have been studied to understand the Jacobian locus inside  the moduli space of principally polarized abelian varieties; see: \cite{GS}.

\subsection{Our results} In this paper we solve the rationality problem for  ${{{\rm{S}}^{{\rm{null}}}_{4}}}$. Indeed it is well-known that the canonical model of a non-hyperelliptic curve $C$ with an effective even theta characteristic is contained in a rank 3 quadric and viceversa a rank 3 quadric containing the canonical model defines a halfcanonical pencil on $C$. In particular for a smooth non-hyperelliptic curve of genus $4$ it can exists at most a unique vanishing thetanull. It is also known that ${\overline{\sM^{{\rm{null}}}_4}}$ is rational \cite[Theorem 3.1]{FL}. This immediately shows that $\Theta_{4,\rm{null}}$ is rational. The study of ${{{\rm{S}}^{{\rm{null}},0 }_{4}}}$ is more complicated due to the fact that $h^0(C,\theta)=0$ 
if $[C,\theta]\in {{{\rm{S}}^{{\rm{null}},0 }_{g}}}$.
\medskip

In this paper we show; see Theorem \ref{rationalityvanishing}:
\begin{thm} The divisor ${\overline{{\rm{S}}^{{\rm{null}},0 }_{4}}}$ of $\overline{{\rm{S^+_4}}}$ is a rational variety. In particular ${\overline{{\rm{S}}^{{\rm{null}},0 }_{4}}}$ is irreducible and reduced.
\end{thm}
%
%
%

Finally consider the Prym moduli space $\sR_4$ which parameterises couples $(C,\eta)$ 
where $\eta$ is a $2$-torsions divisor. By \cite{Ca} we know that $\sR_4$ is rational. For an account on recent progresses in this field see also \cite {V}. Our method implies that 
 the divisor $\sR^{{\rm{null}}}_4$ of $\sR_4$ whose general points are given by those classes $[(C,\eta)]$ such that $C$ admits a unique trigonal linear series and $\eta$ is a nontrivial $2$-torsions divisor is rational:
\medskip 

\begin{thm}
The divisor $\sR^{{\rm{null}}}_4$ is a rational variety.
\end{thm}
\noindent
See Theorem \ref{rationalityprym}.

\subsection{The method}
\label{subsection:the method}

An even  thetacharacteristic $\theta$ is said to be an ineffective theta characteristic if the associated linear system $|\theta|$ is empty and this is the general case. In \cite{TZ1}, \cite{TZ2}, \cite{TZ3} we have built a method to study the coarse moduli space of couples $[(C,\theta)]\in \overline{{\rm{S^+_g}}}$ where $C$ is a trigonal curve, that is $C$ admits a morphism $C\to\mP^1$ of degree $3$ and $\theta$ is ineffective. 

We heavily rely on the empiric observation that Hilbert space of rational curves $R$ of degree $d$ with respect to a given polarization on a (birational) Fano variety $X$
 often give moduli space of curves with an extra data, which typically comes from the intersection theory inside $X$. We learn this general philosophy in \cite{Mu2} and \cite{Mukai12}. Actually we reinterpret it in \cite{TZ1} to the case where $X$ is the del Pezzo threefold $B\subset \mP^6$ and we used it to solve in \cite{TZ2} a long standing Scorza's Conjecture \cite{Sco} and in \cite{TZ3} to show that  $\overline{{\rm{S^+_4}}}$ is birational to $\mP^9$. A variation of the method can be applied to the case where $X$ is the singular del Pezzo threefold with a unique ordinary double point. It enables us to show the rationality of the moduli space of one-pointed ineffective spin hyperelliptic curves; see \cite{TZ4}. 
 
 We think that this quest about the interplay between birational geometry of Fano varities and a new point of view about the description of the moduli space of (spin) curves deserves to be fully explored especially in those cases where complicated objects can be controlled by an easy intersection theory on simple Fano varieties; see \cite{Z} too.

\subsubsection{The method for trigonal spin curves}

In this paper we deep our understanding of the trigonal case where $g=4$. Our construction of trigonal even spin curves relies on the very well-known geometry of the del Pezzo $3$-fold $B$ and of the following diagram; see \cite[\S 2]{FuNa}, c.f. Proposition \ref{prop:FN}:
\begin{equation}
\label{oneuniversal&projection}
\xymatrix{ & \sU_1 \ar[dl]_{{ \pi}}
\ar[dr]^{\varphi}\\
 \sH^B_1=\mP^{2} &  & B\subset \mP^{6} }
\end{equation}
\noindent where $\sH^B_1$ is the Hilbert scheme of lines of $B$, $\sU_1$ is the associated universal family and $\pi\colon\sU_{1}\to \sH^{B}_{1}$, 
$\varphi\colon\sU^{B}_{1}\to B$ are the natural morphisms induced by 
respectively the natural projections $B\times 
\sH^{B}_{1}\rightarrow\sH^{B}_{1}$, $B\times \sH^{B}_{1}\rightarrow 
B$. It is also known that $\sH^B_1$ is isomorphic to $\mP^2$ and  $\sH^B_2$ to $\mP^4$ .

Our construction goes as follows: we start with a simple objet $R_d\subset B\subset\mP^6$, actually $R_d$ is a rational curve of degree $d$, and we construct two related objects $C_d\subset\sU_1$ and $M_d\subset\sH^{B}_{1}$ where $C_d:=\varphi^{-1}(R_d)$ and $M_d:=\pi(C_d)$. We can show that under mild generality assumptions $C_d$ is a smooth curve of genus $d-2$ and the morphism $\pi_{|C_{d}}\colon C_d\to\sH^B_1$ is the morphism $\phi_{|\theta+\delta|}\colon C_d\to\mP^2$ where $\delta$ is the $g^1_3$ given by $\varphi_{|C_{d}}\colon C_d\to R_d$ and $\theta$ comes out from the correspondence on $C_d$ obtained by the lines contained in $B$ which intersect $R_d$ and which mutually intersect. In \cite{TZ2}, we explicitly construct the spin curve $[(C_d,\theta)]\in \overline{{\rm{S^+_{d-2}}}}$. In particular we remind the reader that $M_d$ is a plane curve of degree $d$ with $\frac{(d-2)(d-3)}{2}$ nodes. 
\subsubsection{Sextics of conic type}\label{doveimoduli}More precisely, we define by induction $\sH^B_d$ to be the union 
of the components of the Hilbert scheme
whose general point parameterizes a smooth rational curve of
degree $d$ on $B$ obtained as a smoothing of the union of
a general smooth rational curve of degree $d-1$ belonging to
$\sH^B_{d-1}$ and its general uni-secant line. In \cite{TZ1}, we study the behavior of
the spin curve $[(C_d,\theta)]$ starting from a general $R_d$ of  $\sH^B_d$.

Now consider the case $d=6$. We know that $\sH^{B}_{6}$ is birational to $\mP^6$; see \cite[Theorem 5.2]{TZ3}. Its general points parameterize smooth rational curves $R\subset B$ such that $R$ is of 
degree $6$ inside $\mP^6$ and the linear span $\langle R\rangle$ is $\mP^6$, but it contains general sextics contained in hyperplane sections; see: \cite[Proposition 6.1.1]{TZ3}.  By the method above we construct a rational map $\pi^{+}\colon\sH^B_6\dashrightarrow {\overline{\sS^{+}_{4}}}$ given by $[R_d]\stackrel{\pi^+}{\to} [(C_d,\theta)]$. Since $B$ is a 
$\mP({\rm{SL}}(2,\mC))$-quasi homogeneous variety (see cf. subsection \ref{subsection:geometry of del Pezzo}), we can construct a $\mP({\rm{SL}}(2,\mC))$-action on $\sH^B_6$ and this action is compatible with 
$\pi^{+}\colon\sH^B_6\dashrightarrow {\overline{\sS^{+}_{4}}}$. By taking a resolution of $\pi^{+}$ we can define a compact space ${\widetilde{\sS^{+}_{4}}}$, whose general points parameterize the general $\mP({\rm{SL}}(2,\mC))$ orbits of $\sH^B_6$, 
and two rational maps, $p_{ \sS^{+}_{4} }$ and $q_{\sS^{+}_{4}}$, such that $p_{ \sS^{+}_{4} }\colon \sH^B_6\dashrightarrow{\widetilde{\sS^{+}_{4}}}$ followed by $q_{\sS^{+}_{4}}\colon {\widetilde{\sS^{+}_{4}}}\dashrightarrow {\overline{\sS^{+}_{4}}}$ is the rational Stein factorisation of $\pi^+$; see:  \cite[Corollary 4.16]{TZ3}.  Moreover $q_{\sS^{+}_{4}}\colon {\widetilde{\sS^{+}_{4}}}\dashrightarrow {\overline{\sS^{+}_{4}}}$ is of degree $2$.

In this paper, essentially, we study (part of) the ramification divisor of  $\pi^{+}\colon\sH^B_6\dashrightarrow {\overline{\sS^{+}_{4}}}$. Indeed, by a suitable extension of the method of \cite{TZ1}, we realize that there is a divisor $\sH_{{\rm {c.t.}}}\hookrightarrow \sH^B_6$ such that $\pi^+$ is extendable to an open $\mP({\rm{SL}}(2,\mC))$-invariant subscheme ${\mathring{\sH}}_{{\rm {c.t.}}}$ of $\sH_{{\rm {c.t.}}}$ and $\pi^+(\mathring \sH_{{\rm {c.t.}}})\hookrightarrow {\overline{\sS^{{\rm{null}} }_{4}}}$.  Indeed there is a rich geometry associated to ${\mathring{\sH}}_{{\rm {c.t.}}}$. We prove that if $[R]\in{\mathring{\sH}}_{{\rm {c.t.}}}$ then $R$ is contained in a unique hyperplane section $Z\hookrightarrow B$. Moreover $R$ is given by the general element of the linear system $|2\lambda |$ where $\phi_{|\lambda|}\colon Z\to\mP^2$ is the blow-down of four disjoint lines $\epsilon_i\subset Z\subset B$, $i=1,...4$. Hence $R$ comes with the polarization $\lambda$ such that $\phi_{|3\lambda -\epsilon_1-\epsilon_2-\epsilon_3-\epsilon_4|}\colon Z\to B\cap H$ where $\mP^5=H$ is an hyperplane of $\mP^6$. Due to its shape we call such an element $R\in |2\lambda|$ {\it{a rational sextic of conic type}}. 

\subsubsection{Rational spaces}
Using an explicit identification of the space of smooth polarized hyperplane section $(Z,|\lambda|)$ to an open subscheme $U$ of the Grassmannian ${\rm{G}}(2,5)$ naturally associated to $B$; see Lemma \ref{planes in V}, we show that ${\mathring{\sH}}_{{\rm {c.t.}}}$ is isomorphic to a projective quasi-bundle over $U\subset G(2,5)$; see: Proposition \ref{cleaning}. By this and by the $\mP({\rm{SL}}(2,\mC))$-invariant theory of $B$ we can show that ${\mathring{\sH}}_{{\rm {c.t.}}}//\mP({\rm{SL}}(2,\mC))$ is a rational variety; see Theorem \ref{firstrational}. Then, using the interpretation of the geometry of lines on $B$ as the polar geometry induced on $\mP^2$ by a ${\rm{SO}}(3,\mathbb C)$-invariant conic $\Omega$ we show the reconstruction theorem which says that ${\mathring{\sH}}_{{\rm {c.t.}}}$ dominates ${\overline{\sS^{{\rm{null}},0}_{4}}}$. More precisely, for a general point $[C,\theta]\in{{{\rm{S}}^{{\rm{null}},0 }_{4}}}$  there exists a unique $g^1_3$ on $C$, called $\delta$, hence there is a well defined morphism $\varphi_{|\theta+\delta|}\colon C\to \mP(H^0(C,\sO_C(\theta+\delta ))^\star)=\mP^2$, which is known as the Prym canonical. The intersection theory on $B$ makes possible to see the behaviour of the Prym canonical map, and viceversa; see: the Reconstruction Theorem \ref{perlomodulothree}. We think that the description of $g^1_3$-thetasymmetric curves given in Proposition \ref{symmetry}, the occurrence of the $(4,6)$ lines-points configuration behind the Prym-canonical morphism, see Proposition \ref{threesexstic}, and the Reconstruction Theorem are quite interesting geometrical results in themselves. 

Finally we show that ${\overline{\sS^{{\rm{null}},0 }_{4}}}\hookrightarrow  {\overline{\sS^{+}_{4}}}$ is in the branch loci of $\pi^{+}\colon\sH^B_6\dashrightarrow {\overline{\sS^{+}_{4}}}$; see Proposition \ref{forrationalityvanishing}, and then it follows that it is a rational variety, being birational to ${\mathring{\sH}}_{{\rm {c.t.}}}//\mP({\rm{SL}}(2,\mathbb C))$; see: Corollary \ref{rationalityvanishing}. This result implies the rationality of $\sR^{{\rm{null}}}_4$. Indeed our reconstruction theorem and the geometry of the Prym map give that $\sR^{{\rm{null}}}_4$ is irreducible. Hence we obtain that $\sR^{{\rm{null}}}_4$ is birational to  ${\overline{\sS^{{\rm{null}} }_{4}}}$ since it is easy to find a generically injective rational map  ${\overline{\sS^{{\rm{null}} }_{4}}}\dashrightarrow \sR^{{\rm{null}}}_4$.

\begin{ackn} The author want to thank Hiromichi Takagi, Igor Dolgachev and Ivan Cheltsov for their very useful hints.\end{ackn}

\section{Geometry and invariant theory of the quintic del Pezzo $3$-fold}

The smooth del Pezzo $3$-fold
$B$ is known to be unique up to projective equivalence. We quickly review three different characterisations of  $B$. Indeed to take all these different points of view makes shorter our exposition.

\subsection{The del Pezzo threefold as a complete intersection}
\label{subsection:geometry of del Pezzo}
Let $V$ be a vector space such that ${\rm{dim}}_{\mathbb C}V=5$ and $V^{\vee}:={\rm{Hom}}_{\mathbb C}(V,{\mathbb C})$. Let $W\subset \bigwedge^2 V^{\vee}$ be a $3$-dimensional subvector space; that is $ [W]\in {\mathbb G}(3,\bigwedge^2 V^{\vee})$. We will not always distinguish between the Grassmannian ${\mathbb G}(2,V)$ and its image $\mathbb G$ by the Pl\"ucker embedding inside $\mP(\bigwedge^2 V)$. We set: 
$$B_W:= {\mathbb G}\cap \mP({\rm{Ann}}(W)).$$

We point out the reader that the standard ${\rm{GL}}(V )$-action on $V$ induces a ${\rm{GL}}(V )$-action on  ${\mathbb G}(2,V)$ which is compatible with the standard ${\rm{GL}}(\bigwedge^2 V)$-action.

Note that by definition $\mP({\rm{Ann}}(W))$ is a $6$-dimensional subspace of $\mP(\bigwedge^2 V)$ and by Bertini's theorem it holds that if $W$ is generic then $B_W$ is smooth. More precisely:
\begin{prop}\label{sannazero} 
The complete intersection $B_W$ is a smooth threefold if and only if all forms in $W$ have rank $4$. It also holds that all smooth threefolds $B_W$ are in the same ${\rm{GL}}(V )$-orbit; that is there exists a unique (up to projective equivalence)
 smooth threefold obtained as transversal complete intersection of ${\mathbb G}r(2,V)$ by a six dimensional subspace of $\mP^9$.
\end{prop}
\begin{proof} See c.f. \cite[Lemma 2.1 page 27]{San}.
\end{proof}
From now on we will denote by $B$ the unique smooth threefold obtained as the transversal intersection of ${\mathbb G}(2,V)$ and a $\mP^6$ and $B$ is known as the del Pezzo $3$-fold.

An important result in the theory of del Pezzo's threefolds concerns the image $G'$ inside ${\rm{SL}}(W)\simeq {\rm{SL}}(3,\mathbb C)$ of the stabiliser $SL_W(V)$ of $W$ by the ${\rm{SL}}(V)$-action on ${\mathbb G}(3,\bigwedge^2 V^{\vee})$. Indeed this image is naturally isomorphic to the subgroup 
of ${\rm{SL}}(W)$ given by all transformations fixing a smooth conic $\Omega\subset\mP(W)$, c.f. \cite[Page 28]{San}. Hence $G'\simeq {\rm{SL}}(2,\mathbb C)$.

\subsubsection{An equivariant projection of the Veronese surface}
We remark that there is a natural homomorphism ${\rm{Sym}}^{2}\bigwedge^2 V^{\vee}\to \bigwedge^4 V^{\vee}$ and by the isomorphism 
$\mP(\bigwedge^4 V^{\vee})\simeq \mP(V)$ we can construct a  map:
\begin{equation}\label{squareembedding}
\Box_W\colon \mP(W)\rightarrow \mP(V)
\end{equation}
obtained by the following composition of ${\rm{SL}}(W)$-equivariant morphisms:
$$\mP(W)\stackrel{\nu_2}{\longrightarrow}\mP({\rm{Sym}}^{2}(W))\hookrightarrow \mP({\rm{Sym}}^{2}(\bigwedge^2 V^{\vee}))\to \mP(\bigwedge^4 V^{\vee})\simeq \mP(V).
$$
\begin{lem}\label{squaremorphism} The map $\Box_W\colon \mP(W)\rightarrow \mP(V)$ is an $SL_W(V)$-equivariant embedding.
\end{lem}
\begin{proof} Easy; c.f. see \cite[page 29]{San}.
\end{proof}
We will given later a nice interpretation of the above morphism. Since we will sometimes use it, we call $\Box_W\colon \mP(W)\rightarrow \mP(V)$ the {\it{square embedding}}.

\subsection{$\mP({\rm{SL}}(2,\mathbb C))$-invariant theory description of the del Pezzo threefold}
\label{subsection:geometry of del Pezzo 2}
In order to make easier to follow some of our proofs we present an explicit description of $B$.
\subsubsection{${\rm{SL}}(2,\mathbb C)$-action on binary forms}
Let $V_{d}:=\mC[x,y]_{d}$ be the $d+1$- dimensional vector space of
binary $d$-forms. It is well-known that $V_{d}$ gives a model of the
$d+1$-irreducible representation of ${\rm{SL}}(2,\mC)$ by the
action:
\[
g:= \left(
\begin{array}{cccc}
 a & b\\
 c & d  
\end{array}
\right); \,\, x\stackrel{g}{\mapsto} ax+by;\,\, y\stackrel{g}{\mapsto} cx+dy
\]\noindent
The above action admits a ${\rm{SL}}(2,\mathbb C)$-equivariant split surjection 
\[
s_{l}\colon V_{m}\otimes V_{n}\rightarrow V_{m+n-2l},\,\, 
F\otimes G\stackrel{s_{l}}{\mapsto} (FG)_{l}
\]
where $(FG)_{l}$ is known as the $l$-transvectant of $F$ and $G$; more
explicitly:
$$
(FG)_{l}:= \frac{(m-l)!}{m!}\frac{(n-l)!}{n!}\sum_{i=0}^{l}(-1)^{i}
\left(\begin{array}{c}
 l \\
 i   
\end{array}\right) 
\frac{\partial^{l}F} {\partial x^{l-i}\partial y^{i} }
\frac{\partial^{l}G} {\partial x^{i}\partial y^{l-i}}
$$\noindent By the transvectants it remains defined an isomorphism of
$G$-representations
$$
\bigwedge^{2}V_{d} =
\bigoplus_{l=1}^{[\frac{d+1}{2}]}V_{2(d+1-2l)}¥
$$
where for each $1\leq l\leq [\frac{d+1}{2}]$the projection is given
by: $p_{2l-1}(F\otimes G):=(F,G)_{2l-1}$. We recall that if $l=2k-1$ and $a,b,c,d\in\mC$ then it holds
$$
(aF+bG,cF+dG)_{2k-1}=\det \left(
\begin{array}{cccc}
 a & b\\
 c & d  
\end{array}
\right)\cdot(F,G)_{2k-1}
$$

The set $\{(F,G)_{2k-1}\}_{k}$ is called the set of combinants of $F$ 
and $G$ and actually it is associated to the pencil $\{\lambda F+\mu
G| [\lambda:\mu]\in\mP^1\}$.

\subsubsection{$\mP({\rm{SL}}(2,\mC))$-action on $\mathbb P(\bigwedge^{2}V_4)$.}
Now assume $d=4$. Note that ${\rm{dim}}_{\mathbb C}V_4=5$. From now on we set $G:=\mP({\rm{SL}}(2,\mC))$. The ${\rm{SL}}(2,\mC)$-decomposition $\bigwedge^{2}V_{4}\simeq V_{6}\bigoplus V_{2}$ comes equipped with the two homomorphisms $s_1\colon \bigwedge^{2}V_{4}\to V_6$, $s_3\colon \bigwedge^{2}V_{4}\to V_2$ and it amounts to write
$V_6\simeq {\rm{Ker}}s_{3}$ and $V_2\simeq {\rm{Ker}}s_{1}$. It is easy to find equations for both subspaces of $\bigwedge^{2}V_{4}$. 
In this case the Grassmannian $\mG\subset\mP(\bigwedge^{2}V_{4})$ is the image of the
Grassmannian $G(2,V_{4})$ of the $2$-dimensional subvector spaces of
$V_{4}$ by the Pl\"ucker embedding. Following cf. \cite[ p.505]{MU} and also by Proposition \ref{sannazero} it holds that the del Pezzo 3-fold $B$ can be seen as follows:
$$B=s_{1}(\mG)\subset \mP(V_{6}).$$

By c.f. \cite{FuNa} we also know the $\mP({\rm{SL}}(2,\mC))$-orbit decomposition 
$$
B=G([xy(x^{4}+y^{4}])\cup G([6x^{5}y)]\cup G([x^{6}])
$$
where ${\overline{G([6x^{5}y})]}=G([6x^{5}y)]\cup G([x^{6}])$ is an anticanonical section of $B$.
Moreover by  \cite[Lemma 1.6 p.497]{MU} ${\overline{G([6x^{5}y])}}$ is
singular only along $G([x^{6}])$ and 
${\overline{G([6x^{5}y}])}$ is the image of $\mP^{1}\times\mP^{1}$ by a linear sub-system of bidegree
$(5,1)$. It is well-known that $G([x^{6}])$ is a rational normal sextic of maximal hull.

\subsection{$\mP({\rm{SL}}(2,\mathbb C))$-invariant theory description of lines and conics}
It is well known that $H^i(B,\sO_B)=0$ for $i>0$, that the first Chern-class homomorphism $c_1\colon {\rm{Pic}}(B)\to H^2(B,\mathbb Z)$ is an isomorphism and that ${\rm{Pic}}(B)=[\sO_B(1)]\mathbb Z$ where $\sO_B(1)$ is the very ample invertible sheaf induced by the Pl\"ucker embedding $B\hookrightarrow\mP^6$. Moreover $-K_B=2H$ where $H$ is an hyperplane section. Thanks to the very ample divisor $H$ we call {\it{line}} any subscheme $l\hookrightarrow B$ with $H$-Hilbert polynomial $h_l(t) = 1 + t$ and {\it{conic}} any subscheme $q\hookrightarrow B$ with $H$-Hilbert polynomial $h_q(t) = 1 +2 t$. We denote respectively by $\sH^{B}_{1}$ and $\sH^{B}_{2}$ the Hilbert scheme of
lines and of conics on $B$ for the chosen $H$-polarisation of $B$. Both Hilbert schemes are well-known; see:
\cite{FuNa}, \cite{IlievB5} and cf. \cite{San}.

We revise an explicit description of them using the theory of invariants. We will use freely the identification 
$\bigwedge^{2}V_{4}\simeq V_{6}\bigoplus V_{2}$. We also remind the reader that in this case the dual vector space $V^{\star}_{4}$ of $V_4$ is given by the ring of degree-$4$ homogeneous binary partial differential operators 
$\mathbb C[\partial_{x},\partial_y]_{4}=V^{\star}_{4}$ and accordingly we have the ${\rm{SL}}(2,\mathbb C)$-splitting $\bigwedge^{2}V_{4}^*\simeq V_{6}^*\bigoplus V_{2}^*$

\subsubsection{Invariant theory description of $\sH^{B}_{1}$.}
In the flag variety $\sF(1,3, V_{4})$ we consider the following subscheme: 
$$
\sH_{1}:=
\{[\langle\alpha\rangle\subset\langle\alpha,\beta_{0},\beta_{1}\rangle]\in
\sF(1,3, V_{4})\mid s_{3}(\alpha\wedge\beta_{0})=0,
s_{3}(\alpha\wedge\beta_{1})=0 
\}
$$
\noindent We define the homomorphism:
 $$
 \sigma_3\colon \sH_{1}\to \mP(V_{2}), \,\,[\langle\alpha\rangle\subset\langle\alpha,\beta_{0},\beta_{1}\rangle] \mapsto [ s_{3}(\beta_0\wedge\beta_{1})].
 $$\noindent
 It is straightforward to show that :

\begin{prop}\label{lines} There exists a $G$-equivariant isomorphism between $\sH_{1}$ and $\sH^{B}_{1}$. Moreover the natural homomorphism 
$\sigma_{3}\colon \sH_{1}\rightarrow\mP(V_{2})$ is a  $G$-equivariant isomorphism which induces a $G$-equivariant
    isomorphism $\sigma_{3}\colon \sH^{B}_{1}\rightarrow\mP(V_{2})$.
\end{prop}
\begin{proof} See c.f. \cite[Proposition 2.20]{San}.
\end{proof}

By  \cite{MU} we know that the anticanonical section ${\overline{G([6x^{5}y])}}$ is the loci swept by the lines $l\subset B$ such that
$\sN_{l/B}=\sO_{l}(-1)\oplus\sO_{l}(1)$. These lines are very important to understand the geometry of $B$. Consider $\phi_{1}\colon\sH^{B}_{1}\rightarrow\mP(V_{4})$ the morphism
induced by the projection $\sF(1,3, V_{4})\rightarrow \mP(V_{4})$; it is easy to show the following:

\begin{cor}\label{veroneseproj} The locus
    $\Omega':=G([x^{2}]):=\{{}^{\tau}[a_{0},a_{1},a_{2}]\in\mP^2\mid a^{2}_{1}-a_{0}a_{2}=0\}$is a conic inside $\mP(V_{2})$ and it parameterizes the loci of lines
    $l\subset B$ such that $\sN_{l/B}=\sO_{l}(-1)\oplus\sO_{l}(1)$.
    Moreover the composition $\phi_{1}\circ \sigma_{3}^{-1}\colon
    \mP(V_{2})\rightarrow \mP(V_{4})$ is induced by the square morphism.
\end{cor}
\begin{proof} It is an easy interpretation of the square morphism given in Lemma \ref{squaremorphism} where $V_2^*=W\subset \bigwedge^2V_4^*$, and $V=V_4$ and the isomorphism $V_{4}\simeq \bigwedge^{4}V_{4}^*$.
\end{proof}
Hence we can switch from the coordinate free presentation of subsection (\ref{subsection:geometry of del Pezzo}) to the binary $G$-invariant theory, by the following dictionary: $V=V_4$, $W:=V_2$ where $V_2$ is seen as a subspace of $\bigwedge^2 V_4$ and $B_W=\mP(V_6)\cap{\mathbb G}$, $\sH^{B}_{1}=\mP(V_2)$, $\Omega'=\Omega$.

\subsubsection{Invariant theory description of $\sH^{B}_{2}$.}
To simplify notation we denote by $K_{\beta}:=
{\rm{Ker}}(s_{3}(\beta\wedge \ldots))\colon V_{4}\rightarrow V_{2}$.
where $\beta\in V_{4}$.
Since $V^{\star}_{4}$ is the ring of homogeneous binary partial differential operators of degree $4$,
$\mathbb C[\partial_{x},\partial_y]_{4}=V^{\star}_{4}$, then $V^{\star}_{4}\simeq \bigwedge^{4}V_{4}$. We recall that in general ${\rm{dim}}
K_{\alpha}=2$ but if $\alpha\in  SL_{2}(x^{4})\cup SL_{2}(x^{2}y^{2})$
then  ${\rm{dim}}K_{\alpha}=3$. Now, in the flag variety $\sF(2,4,V_{4})\subset 
\mP(\bigwedge^{2} V_{4})\times\mP(\bigwedge^{4}V_{4})$, we define
\[\sH_{2}:=
\{([U],[W]\in \sF(2,4,V_{4})
\mid U= \langle\alpha,\beta\rangle, W=\langle
K_{\alpha},K_{\beta}\rangle,   {\rm{dim}}_{\mathbb C}U=2, {\rm{dim}}_{\mathbb C}W=4\}.
\]

It is a nice exercise in invariant theory to show that:
\begin{prop}\label{coniche} The Hilbert space $\sH^{B}_{2}$ of conics of $B$ is $G$-isomorphic to $\sH_{2}$. Moreover the morphism $\sH_{2}\to \mP(V_4^\star)$ given by the natural projection is an isomorphism which induces a $G$-isomorphism $\sH^{B}_{2}\simeq \mP(V^{\star}_{4})$
\end{prop}
\begin{proof} See c.f. \cite[Prop. 2.32]{San}.
\end{proof}

\begin{rem}The locus of double lines inside $\mP(V^{\star}_{4})$ is
    given by the $G$-orbit of $[\frac{\partial^{4}}{\partial x^{4}}]$.
    In particular it is a rational normal curve of degree $4$.
\end{rem}

\subsubsection{Invariant theory description of $\sH^{B}_{3}$.}

A subscheme $\Gamma$ of $B$ whose $H$-Hilbert polynomial is $3t+1$ is called a rational cubic. Indeed by \cite[Corollary 1.39]{San} we know that such subschemes have no embedded points and satisfy $h^1(\Gamma,\sO_\Gamma)=0$. We denote by $\sH^{B}_{3}$ the corresponding Hilbert scheme. We need an explicit description of $\sH^{B}_{3}$.

Let $\sF$ be the restriction to $B\subset G(2,V)$ of the dual of the universal bundle of rank two on $G(2,V)$; the reader may think $V$ as $V_4$; see section \ref{subsection:geometry of del Pezzo}. The projective bundle $\mP(\sF)\subset B\times \mP(V)$ is the family of lines in $\mP(V)$ parameterized by $B$. In \cite{C} Castelnuovo showed that
$B$ parameterizes tri-secant lines of the projected Veronese surface $\mP^2\simeq P\subset \mP(V)= \mP^4$; see: c.f. \cite[Lemma 7.1]{PV} or c.f \cite[Corollary 3.12]{San}. By Corollary \ref{veroneseproj} $P$ is the image of $\mP(V_2)$ inside $\mP(V_4)$ given by the square morphism of 
Lemma \ref{squaremorphism}. Hence we have a non obvious nice interpretation of the natural diagram:

\begin{equation}
\label{tris&line}
\xymatrix{ & \mP(\sF)\ \ar[dl]_{ \pi_{B}}
\ar[dr]^{\pi_{\sF}   }\\
 B&  &  \mP(V)}
\end{equation}
\begin{lem}\label{castelnuovo} The natural projection $\pi_{\sF} \colon \mP(\sF)\to \mP(V)$ is the $\mathbb P({\rm{SL}}(2,\mathbb C))$-invariant blow-up along $P$.
\end{lem}
\begin{proof} Indeed by \cite[Lemma 7.1]{PV} $\pi_{\sF} \colon \mP(\sF)\to \mP(V)$ is one to one outside $P$ and the fiber over a point of $P$ is $\mP^1$.
Thus it is the blow-up along $P$ since both $\mP(\sF)$ and $\mP(V)$ are smooth fourfolds. All the maps are $\mathbb P({\rm{SL}}(2,\mathbb C))$-equivariant if we identify $V=V_4$.
\end{proof}

We briefly review \cite[Proposition 2.46]{San}, \cite[Section 3 page 57]{San}, \cite[Proposition 3.16]{San} . Let $E_{\sF}$ be the $\pi_{\sF}$-exceptional divisor. Let $l\subset \mP(V)$ be a line and $l'\subset \mP(\sF)$ its strict transform. Then $-K_{\mP(\sF)}=\pi^*\sO_{\sF}(5)-E_{\sF}$.
Thus, if $P\cap l=\emptyset$, then $-K_{\mP(\sF)}\cdot l'=5$.
Let $T_{\sF}$ be the tautological divisor on $\mP(\sF)$, which is the pull-back of 
$\sO_{\mP(V)}(1)$; see: \cite[Proposition 3.10]{San}. Then $T_{\sF}\cdot l'=1$. Let $\pi_B\colon  \mP(\sF)\to B$ be the natural projection and $\sO_{\mP(\sF)}(H')$ be the $\pi_B$-pull-back of $\sO_B(1)$.
Then by the canonical bundle formula for $\mP^1$-bundle $-K_{\mP(\sF)}=2T_{\sF}+H'$ 
since $-K_B=2H$ and since we know by the standard relative tautological sequence for $\pi_B^\star(\sF)$ that $\det \sF=\sO_B(1)$. Therefore $H\cdot l' =5-2=3$.
Thus the image of $l'$ on $B$ is a twisted cubic. In case $P\cap l\not =\emptyset$, $l$ corresponds to a degenerate twisted cubic. More precisely if $l$ is unisecant to $P$ then its pull back is ${\widehat{l}}+\zeta$ where $\zeta\subset E$ is the $\mP^1$ inside $\mP(\sF)$ such that $\pi_{\sF}(\zeta)$ is the unique point $P\cap l$. Hence $H\cdot \pi_{B}({\widehat{l}})=2$ and $H\cdot \pi_{B}(\zeta)=1$; that is $\pi_{B}(\zeta)$ is a line which intersects the conic $\pi_{B}({\widehat{l}})$ but obviously they are not coplanar since $B$ is an intersection of quadrics.
 If $l$ is bisecant to $P$ then  its pull back is ${\widetilde{l}}+\zeta+\zeta'$ and the image $\pi_{B}({\widetilde{l}})$ is a line. Finally if $l$ is a trisecant line then its pull back is $l^{0}+\zeta+\zeta'+\zeta''$ and $\pi_{B}(l^{0})$ is the common point on $B$ of the three lines $\pi_{B}(\zeta), \pi_{B}(\zeta') ,\pi_{B}(\zeta'')$. There are also some degenerations depending on degenerate intersection between $P$ and $l$. In any case the crucial point is that no line is contained inside $P$.

\begin{lem}\label{lines in V} The map $\mathbb G(2,5)\to \sH^B_3$ induced by $[l] \mapsto [\pi_{B\star}\pi_{\sF}^{\star}(l)]$ is a $\mathbb P({\rm{SL}}(2,\mathbb C))$-isomorphism.
\end{lem}
\begin{proof} See the above discussion. A formal proof is given in \cite[Proposition 2.46]{San}. \end{proof}

\subsubsection{The scheme of polarised hyperplane sections.}

In the same vein and for later use we need to have a rough description of an open set of the scheme of polarised hyperplane sections of the del Pezzo threefold $B$.

Following \cite{San} we can define $D_{{\rm{tri}}}$ to be the divisor of those $[K]\in \mathbb G(3,V)$ such that $K$ contains a line trisecant to $P$. By 
\cite[Lemma 3.28]{San} we know that for every $K\in  \mathbb G(3,V)\setminus D_{{\rm{tri}}}$ the morphism $\pi_K\colon S_K\to B_K$ is an isomorphism where $S_K$ is the $\pi_{\sF}$-strict transform of $K$ and $\pi_K\colon S_K \to B_K$ is the morphism induced by the restriction to $S_K$ of $\pi_B\colon \mathbb P(\sF)\to B$. Set 
\begin{equation}\label{openopen}\mathring{\mathbb G}:=\{[K]\in  \mathbb G(3,V)\mid K_{\mid P}\, {\rm{are\, 4\, distinct \, points}}\}.
\end{equation}
Obviously $\mathring{\mathbb G}\hookrightarrow  \mathbb G(3,V)\setminus D_{{\rm{tri}}}$ is an open embedding. We can define the morphism
\begin{equation}\label{alpha}
\alpha\colon\mathring{\mathbb G}  \to \check\mP^6,\, [K] \mapsto [H_K]
\end{equation}
where, letting $B_K=\pi_{B\star}\pi_{\sF}^{\star}(K)$, $H_K$ is the unique hyperplane of $\mP^6$ such that $B_{|H}=B_K$. Let 
$$
 \mathring{\sK}:=\{([K],b,[H])\in\mathring{\mathbb G}\times B\times\check\mP^6|, B_{|H}=B_K, \, {\rm{and}}\,  b\in H\cap B\}
 $$
\begin{lem} \label{planes in V} 
$\mathring{\mathbb G}$ is isomorphic to the parameter space of points $([H], \sO_Z(\lambda))$ where $[H]\in\check\mP^6$, $Z=H\cap B$ and $\sO_Z(\lambda)$ is an invertible sheaf on $Z$ such that $\phi_{|\lambda|}\colon Z\to \mP^2$ is the blow-down of $4$ rational curves. Moreover  
$\alpha\colon  \mathring{\mathbb G} \to \check\mP^6$ is compatible with the natural forgetful morphism  $\mathring{\sK}\to\check\mP^6$.
\end{lem}
\begin{proof} Since the fibers of the natural projection $\mathring{\sK}\to \mathring{\mathbb G}$ are irreducible $\alpha$ factorise the Stein factorisation of the natural projection $\mathring{\sK}\to \check\mP^6$. We have only to show that given a smooth section $Z$ of $B$ inside $\mP^6$ and a polarisation $\lambda$ there exists $[K]\in \mathring{\mathbb G}$ such that $Z=B_K$. Indeed to give such a $\lambda$ is equivalent to give four disjoint lines $e_1, e_2,e_3,e_4\subset Z$ such that $\phi_{|\lambda|}\colon Z\to \mP^2$ is the blow-up at four points $a_i$ where $e_i=\phi_{|\lambda|}^{-1}(a_i)$, $i=1,2,3,4$. To give such four points 
$[e_1],[e_2],[e_3],[e_4]\in\sH^B_1=\mP(W)$ is equivalent to give a pencil of conics inside $B$ that is to give a line inside $\mP(V^*)$ via its identification to $\sH^B_2$ given in Proposition \ref{coniche}. By construction of the square morphism and by the interpretation of it for the geometry of $B$, the base locus of the corresponding pencil of hyperplanes is the claimed plane $K\subset \mP(V)$.
\end{proof}

Note that we do not know if any element $[K]\in \mathbb G(3,V)$ can be seen as a couple $([H], \sO_Z(\lambda))$ where 
$[H]\in\check\mP^6$, $Z=H\cap B$ and $\sO_Z(\lambda)$ is an invertible sheaf on $Z$. In any case we are interested in a rationality problem and the standard $\mathbb P{\rm{SL}}(2,\mathbb C)$-action on $B$ extends to our construction in a equivariant way. We recall that we have set $G= \mathbb P{\rm{SL}}(2,\mathbb C).$

 \begin{lem}\label{quotient grassmannian} It holds:
 \begin{enumerate}
 \item$\mathring{\mathbb G}$ is a $G$-invariant open subscheme of $\mathbb G(3,V)$. 
 \item The $[K]$-stabiliser is trivial if $[K]\in \mathring{\mathbb G}$ is a general element.
 \item $\mathring{\mathbb G}//G$ is birational to $\mP^3$.
 \end{enumerate}
 \end{lem}
 \begin{proof} The first claim is easy. To show the second one we can be back to the interpretation given in Corollary \ref{veroneseproj}. In other words we can consider the standard $G$ action on binary forms and assume that $K\cap P=\{[\alpha_1^2],...,[\alpha_4^2]\}$ where $\alpha_i\in V_2$. In particular w.l.g we can assume that there exists $\lambda_i\in \mathbb C$ such that
 \begin{equation}\label{equazionequadri}
 \alpha_4^2=\sum_{i=1}^3\lambda_i\alpha_i^2
 \end{equation}
 Now if there exists an element $[g]\in G$ which stabilises $K$ then $g$ has to permute the $\alpha_i$'s. This is clearly impossible since they are $3$ by $3$ independent and any other relation among them similar to Equation (\ref{equazionequadri}) gives a contradiction.

 To show the third one is the same as to show that $\mathbb G(3,V_4)//G$ is rational. Since  $G(3,V_4)$ is $G$-equivariantly isomorphic to $G(2,V_4^{\star})$ with the dual  $G$-action we are reduced to show that  $G(2,V_4)// G$ is birational to $\mP^3$ where $G$ is induced by the standard ${\rm{SL}}(2,\mathbb C)$-action on $V_4$. Any element $[U]\in G(2,V_4)$ is given by $U=\langle f,g\rangle$ where $f,g$ are two linearly independent quartic binary forms. By $G$-action we can take $f(x,y)=xy(x-y)(\lambda y-\mu x)$ where $[\lambda:\mu]\in\mP^1$. Set 
  $g:=\sum_{i=0}^{4}a_{i}
\left(\begin{array}{c}
 4 \\
 i   
\end{array}\right) x^{4-i}y^{i}$. Then we form the $5\times 5$ anti-symmetric matrix of the element $[f\wedge g]\in \mP(\bigwedge^{2}V_4)$. Finally we set $t:=\lambda/\mu$, $a_0=1$ and we impose the vanishing of the $5$ pfaffians according to the basic rules of \cite[ p.505]{MU}. 
Now it is very easy to show that $a_4$ and $a_3$ depends rationally on the free variables $t, a_1,a_2$.
\end{proof}

\subsection{ The del Pezzo threefold as the variety of sums of a conic}
There is another neat description of $B$ which makes transparent the meaning of the invariant conic $\Omega'$. 
\subsubsection{Polarity with respect to a conic and lines of  the del Pezzo $3$-fold}
 Let $\{\check{F}_2=0\}\subset \check{\mP}^2$ be
a smooth conic. Set
\[
\VSP(\check{F}_2,3)^o:=
\{([H_1], [H_2], [H_3])\mid H_1^2+H_2^2+H_3^2=\check{F}_2\}
\subset \Hilb^3 {\mP}^2,
\]
where ${\mP}^2$ is the dual plane to $\check{\mP}^2$,
thus linear forms $H_i$ of  $\check{\mP}^2$ ($i=1,2,3$) can be considered as
points in ${\mP}^2$. 
Mukai showed in \cite{Mu2} that
$B$ is isomorphic to 
the closed subset 
$\VSP(\check{F}_2,3):=\overline{\VSP(\check{F}_2,3)^o}\subset 
\Hilb^3 {\mP}^2$
where this ${\mP}^2$ is isomorphic to $\sH^B_1$. The variety $\VSP(\check F_2,3)$ has the natural action of
the subgroup ${\rm{SO}}(3,\mathbb C)$ of the automorphism group $\PGL_3$ of 
${\mP}^2$
consisting of elements which preserve $\{\check{F}_2=0\}$.
This group is isomorphic to $G$,
and the conic is the unique one which is $G$-invariant.
By definition of $\VSP(\check{F}_2,3)^o$, it is easy to see that
$G$ acts on $\VSP(\check{F}_2,3)^o$ transitively.
Thus we recover also in this way that $B$ is a quasi-homogeneous $G$-variety.
Actually, it is proved in \cite{PV} that $G$ is the automorphism group of $B$. Moreover, Mukai showed that,
for a point $b:=([H_1],[H_2], [H_3])\in \VSP(\check{F}_2,3)^o
\subset B$, the points
$[H_i]\in {\mP}^2$ ($i=1,2,3$) represent three lines through $b$.
By definition of $\VSP(\check{F}_2,3)^o$ and by the
transitivity of the action of $G$ on $\VSP(\check{F}_2,3)^o$, 
it is easy to show the following claim:
\begin{cla}
\label{cla:doubly}
$G$ acts transitively 
on the set of unordered pairs of intersecting lines
whose intersection points are contained in 
$\VSP(\check{F}_2,3)^o$.
\end{cla}

Let ${F}_2$ be the quadratic form dual to $\check{F}_2$ and set
\[
\Omega'':=\{{F}_2=0\}
\]
for the associated conic in 
${\mP}^2$.
The conic $\Omega''\subset {\mP}^2$
is the unique one
invariant under the induced action of $G$ on $\sH^B_1$.
Moreover,
$G$ is exactly the closed subgroup of $\Aut {\mP}^2\simeq \PGL_3$
whose elements preserve $\Omega''$. 

In Corollary \ref{veroneseproj}, see also \cite[\S 2.5]{I}, we have shown that
there exists a conic $\Omega'$ in ${\mP}^2$
such that,
for $[l]\in {\mP}^2- \Omega'$ (resp. 
for $[l]\in \Omega'$),
it holds $\sN_{l/B}=\sO_{l}\oplus\sO_{l}$
(resp. $\sN_{l/B}\simeq 
\sO_{\mP^1}(-1)\oplus \sO_{\mP^1}(1)$).
Obviously $\Omega'$ is invariant under the action of $G$,
hence we have $\Omega'=\Omega''$. Hence $\Omega=\Omega''$. From now on we will not distinguish between $\Omega',\Omega''$ and $\Omega$.

\begin{defn}
A line $l$ on $B$ is called a {\em{special line}} if
$\sN_{l/B}\simeq 
\sO_{\mP^1}(-1)\oplus \sO_{\mP^1}(1).$
\end{defn}

The following Proposition is treated in \cite[4.2]{D1} and it will play a fundamental role in the Reconstruction Theorem; see: Theorem \ref{perlomodulothree}.
\begin{prop}
\label{prop:polar}
Let $\widetilde{\Omega}$ be the symmetric bi-linear form
associated to $\Omega$. Then 
two lines $l$ and $m$ on $B$ intersect if and only if 
$\widetilde{\Omega}([l],[m])=0$. In particular there is a $G$-identification between $(\sH^B_1,\Omega)$ and $(\mP^2, Q)$ where $Q$ is a smooth conic.
\end{prop}

\begin{defn}\label{polarelinea} Two points $[l],[m]\in\sH^B_1$  are said to be {\em polar} with respect to
$\widetilde{\Omega}$ if $\widetilde{\Omega}([l],[m])=0$. 
\end{defn}

\begin{rem} In other words two lines $l,m\subset B$ are incident iff the corresponding points $[l],[m]\in\sH^B_1$  are polar with respect to
$\widetilde{\Omega}$.
\end{rem}

\subsection{The $3$-to-$1$ cover of $B$ given by the universal family of lines}

 Our construction of spin curves relies on the geometry of the following diagram which has been deeply studied in \cite[\S 2]{FuNa}:
\begin{equation}
\label{universal&projection}
\xymatrix{ & \sU_1 \ar[dl]_{{ \pi}}
\ar[dr]^{\varphi}\\
 \sH^B_1=\mP^{2} &  & B\subset \mP^{6} }
\end{equation}
\noindent where we denote by $\pi\colon\sU_{1}\to \sH^{B}_{1}$ and by 
$\varphi\colon\sU^{B}_{1}\to B$ the natural morphisms induced by 
respectively the natural projections $B\times 
\sH^{B}_{1}\rightarrow\sH^{B}_{1}$, $B\times \sH^{B}_{1}\rightarrow 
B$.

\begin{nota}\label{nota}
For an irreducible curve $R$ on $B$, 
denote by $M(R)$ the locus $\subset \sH^B_1$ of lines intersecting $R$,
namely, $M(R):=\pi(\varphi^{-1}(R))$ with reduced structure. 
Since $\varphi$ is flat, $\varphi^{-1}(R)$ is purely one-dimensional.
If $\deg R\geq 2$, 
then $\varphi^{-1}(R)$ does not contain a fiber of 
$\pi$, thus $M(R)$ is a curve.
See Proposition \ref{prop:FN} for the description of $M(R)$
in case $R$ is a line of $B$.  
\end{nota}

\begin{prop} 
\label{prop:FN}
It holds$:$ 
\begin{enumerate}[$(1)$]
 \item
the union of special lines is
the branched locus $B_{\varphi}$ of 
 $\varphi\colon \mP\to B$.
$B_{\varphi}$ has the following properties:
\begin{enumerate}[$({1}\text{-}1)$]
\item $B_{\varphi}\in |-K_{B}|$, 
\item $\varphi^*B_{\varphi}=B_{\varphi,_1}+2B_{\varphi, 2}$,
where $B_{\varphi,_1}\simeq B_{\varphi,_2}\simeq \mP^1\times \mP^1$, and
$\varphi\colon B_{\varphi,_1}\to
B_{\varphi}$ and 
$\varphi\colon B_{\varphi,_2}\to
B_{\varphi}$ are injective, and
\item
the pull-back of a hyperplane section of
$B$ to $B_{\varphi,_1}$ is a divisor of type $(1,5)$,
\end{enumerate}
\item the image of $B_{\varphi,_2}$ by $\pi\colon \mP\to \sH_{1}^{B}$
is the conic $\Omega$, 

\item
if $l$ is a special line,
then $M(l)$ is the tangent line to $\Omega$ at $[l]$.
If $l$ is not a special line,
then $\varphi^{-1}(l)$ is the disjoint union of 
the fiber of $\pi$ corresponding to $l$,
and the smooth rational curve dominating a line on $\sH^B_1$.
In particular,
$M(l)$ is the disjoint union of a line and the point $[l]$.

{\em{By abuse of notation}}, we denote by $M(l)$ the one-dimensional part of
$M(l)$ for any line $l$. 
Vice-versa,
any line in $\mathcal{H}_{1}^{B}$ is of the form $M(l)$ for some line $l$,
and
\item 
the locus swept by lines intersecting $l$ is a hyperplane section $T_{l}$
  of $B$ whose singular locus is $l$. For every point $b$ of $T_l- l$,
  there exists exactly one line which belongs to $M(l)$ 
  and passes through $b$.

\end{enumerate}
\end{prop}
\begin{proof} See \cite[\S 2]{FuNa} and \cite[\S 1]{IlievB5}. 
\end{proof}

\section{Special rational curves on the quintic del Pezzo $3$-fold}
\subsection{Rational sextics of conic type}
\label{subsection:the three cases}
In the Introduction we have recalled our theory of genus-$4$ spin curves via rational sextics on $B$; see \cite{TZ3}. In this paper we are interested in special loci inside ${\overline{S^{+}_{4}}}$. 
Hence we need to study special loci in the Hilbert scheme of rational sextics of $B$.
\subsubsection{Some known results on the Hilbert space of rational curves on B}
For a smooth projective variety $X$ in some projective space,
let $\sH^{0}_{d}(X)$ be the union of components
of the Hilbert scheme whose general points parameterize
smooth rational curves on $X$
of degree $d$.
%
Let $\sH^{0'}_{d}(X)$ be 
the open subset of $\sH^{0}_{d}(X)$
parameterizing smooth rational curves on $X$
of degree $d$ with linear hull of maximal dimension.

We define by induction $\sH^B_d$ to be the union 
of the components of the Hilbert scheme
whose general point parameterizes a smooth rational curve of
degree $d$ on $B$ obtained as a smoothing of the union of
a general smooth rational curve $R$ of degree $d-1$ contained in
$\sH^B_{d-1}$ and a general uni-secant line to $R$ contained in $B$.

We know that for some $d$, $\sH^{B}_{d}$ contains elements $[R]$ such 
that $\langle R\rangle$ is not of maximal dimension and we can show inductively that
$\sH_d^{B}\subset \overline{{\sH}^{0'}_{d}(B)}$, 
where we take the closure in the Hilbert scheme; see:\cite{TZ1}. It is known that $\sH^{B}_{d}$ is a rational variety if $d\leq 5$. 
We also know that:
\begin{prop}
\label{prop:Cdbis} 
${\sH}_d^{B}=\overline{\sH^{0'}_{d}(B)}$ for $d\leq 6$. $\sH^B_6$ is the closure of the Hilbert scheme of
sextic normal rational curves on $B$, and it is a rational variety of dimension $12$.
\end{prop}
\begin{proof} See\cite[Corollary 3.10 and Theorem 5.1]{TZ3}. \end{proof}

\subsubsection{Rational curves of degree $6$}
Let $R\subset B$ be a smooth rational curve of degree $6$. It is known by the invariant theory description of $B$ that there exists a smooth rational sextic inside $B$ which is contained inside no hyperplane sections. 
On the other hand an easy degree 
count shows that if $R$ is contained inside a hyperplane section $Z$ of $B$ then $Z$ is the unique one. In other words if $R\subset B$ then for its linear span it holds:
$\langle R \rangle =\mP^{6}$ or $\langle R \rangle =\mP^{5}$. Let $Z:=\langle R\rangle \cap B$. If $Z$ is smooth then $Z$ is a smooth quintic del Pezzo surface. Then there exists an isomorphism  
${\rm{Bl}}_{a_{1},a_{2},a_{3},a_{4} }\mP^{2}\rightarrow Z$ where 
${\rm{Bl}}_{a_{1},a_{2},a_{3},a_{4}}\mP^{2}\to \mP^2$ 
is the blow-up of $\mP^{2}$ at four
points $a_{1},a_{2},a_{3},a_{4}$. 
Denote by $\lambda$ the pull-back of a line of $\mP^{2}$ and let
$e_i$ where $1\leq i \leq 4$ be the corresponding four exceptional curves. In the following we will 
often omit to distinguish between $Z$ and 
${\rm{Bl}}_{a_{1},a_{2},a_{3},a_{4}}\mP^{2}$ if no danger of confusion can 
occur.

\begin{lem}\label{threecases}
    Let $R\subset B_{}$ be a smooth rational curve of degree $6$.
    Assume that $\langle R\rangle \cap B$ is a smooth subvariety. 
    Then, up to Cremona transformations, only the following cases do occur:
\begin{enumerate}[\rm(I)]
\item
$\langle R \rangle\simeq \mP^6$, that is $\langle R\rangle \cap
B=B_{}$;
\item $\langle R \rangle\simeq \mP^5$ and 
$R\sim 5\lambda-2(e_1+e_2+e_3)-3e_4$ on $Z$;
\item $\langle R \rangle\simeq \mP^5$ and $R\sim 2\lambda$ on $Z$.
\end{enumerate}
\end{lem}
\begin{proof} We only need to recall that
    $|3\lambda-e_{1}-e_{2}-e_{3}-e_{4}|$
    is the linear system which embeds 
    the blow-up $Z$ of $\mP^2$ at  $a_{1},a_{2},a_{3},a_{4}$ into $\mP^{5}$. The
rest 
    is an easy computation on cycles on $Z$. We only stress that
    numerically it occurs also the case $\langle R \rangle\simeq\mP^5$ and 
$R\sim 4\lambda-e_1-e_2-e_3 -3e_4$ on $Z$; but this is the same of ${\rm{(II)}}$
up to a Cremona transformation.
\end{proof}

\begin{prop}\label{gradoduno6} 
The locus whose general point $[R]$ represents
     a general sextic rational curve of type ${\rm{(III)}}$ contained in a 
    general hyperplane section of $B$ is a divisor of $\sH^B_6$
  \end{prop}
  \begin{proof} By Proposition \ref{prop:Cdbis}  $\sH^B_6$ is irreducible of dimension $12$. By Lemma \ref{threecases} the claim follows by an easy application of Riemann-Roch theorem on smooth quintic del Pezzo surfaces.
  \end{proof}

\begin{nota}\label{notaH}
From now on we will denote by $\sH_{\rm{{c.t.}}}$ the divisor in $\sH^{B}_{6}$ which is the closure of the loci whose general element $[R]$ is a rational sextic as in ${\rm{(III)}}$ of Lemma \ref{threecases}. Its irreducibility will be proven later.
\end{nota}

\begin{defn}
We call a smooth sextic $R\subset B$ such that $[R]\in\sH_{\rm{{c.t.}}}$ a {\it{sextic of conic type}} and $\sH_{\rm{{c.t.}}}$ is accordingly called {\it{the Hilbert scheme of rational sextics of conic type}}.
\end{defn}

For later use we prove:
\begin{prop}
\label{prop:Cd1III}
Let $[R]$ be a general smooth sextic of conic type.
Then it satisfies the following conditions:
\begin{enumerate}[$(1)$]
\item
there exist no $k$-secant lines of $R$ on $B$ with $k\geq 3$,
\item there exist exactly six bi-secant lines of $R$ on
$B$, and any of them intersects $R$ simply, and
\item $R$ intersects transversally $B_{\varphi}$.

\end{enumerate}
\end{prop}
\begin{proof} By lemma \ref{threecases} $R$ is contained in a smooth hyperplane section $Z$ of $B$ and $B_{\varphi|Z}\in |6\lambda-2\sum_{i=1}^{4}e_i|$, then it is easy to see that the general element $R\in |2\lambda|$ is transversal to $B_{\varphi|Z}$ and that its $6$ bisecants are ordinary ones and they  intersect transversally $B_{\varphi}$.
\end{proof}

\subsection{An open subscheme of the Hilbert scheme of sextics of conic type}
To control the $G$-action on $\sH_{\rm{{c.t.}}}$ we need to have an explicit description of this action.
Actually for our rationality theorem it will be sufficient to give an interpretation of a $G$-invariant open subset of $\sH_{\rm{{c.t.}}}$ in terms of flag varieties naturally associated to the presentation of $B$ as a subvariety of $G(2,5)$.

 By Lemma \ref{threecases} $(3)$ we see that $[R]\in\sH_{\rm{{c.t.}}}$ comes with a complete linear system $|\lambda|$ on $Z$ such that $R\in |2\lambda|$ and $\phi_{|\lambda|}\colon Z\to\mathbb P^2$ is the contraction of $4$ rational curves $e_1,e_2,e_3,e_4$. We can construct a $G$-equivariant compactification of parameter spaces of triples $([H], |\lambda|, R)$ where $[H]\in\check\mP^6$, $Z=H\cap B$ is smooth, $R\in |2\lambda|$, and $\sO_Z(\lambda)$ is the pull-back of the line bundle $\sO_{\mP^2}(1)$.

Consider $\mathring \sH_{\rm{{c.t.}}}\hookrightarrow\sH_{\rm{{c.t.}}}$ the open subscheme of $\sH_{\rm{{c.t.}}}$ given by those $[R]\in\sH_{\rm{{c.t.}}}$ such that $R$ 
is smooth and the hyperplane section $\langle R\rangle\cap B=Z$ is smooth. It remains defined a natural morphism $\mathring \epsilon\colon \mathring \sH_{\rm{{c.t.}}}\to\check\mP^{6}$ given by $[R]\stackrel{\mathring \epsilon}{\rightarrow} [\langle R\rangle]$. Let $\mathring \sigma\circ\mathring\rho=\mathring \epsilon$ be the Stein factorization of $\mathring\epsilon$ where $\mathring\rho\colon \mathring \sH_{\rm{{c.t.}}}\to\mathring \sT$ has irreducible fibers and $\mathring \sigma\colon \mathring \sT\to \check\mP^{6}$ is a finite covering over the open set given by the hyperplane sections which are transversal to $B$. By construction $\mathring \rho\colon \mathring \sH_{\rm{{c.t.}}}\to\mathring \sT$ is a quasi-$\mP^5$-bundle over $\mathring \sT$. Up to now we do not know if $\mathring \sT$ is irreducible. Hence we do not yet know that $\mathring \sH_{\rm{{c.t.}}}$ is irreducible. The goal is to have an explicit construction of a suitable $G$-equivariant compactification $\sT$ of $\mathring \sT$. By definition we only have that the inclusion $\mathring \sH_{\rm{{c.t.}}}\hookrightarrow\sH_{\rm{{c.t.}}}$ is open.

\subsubsection{The flag variety $F(2,3,5)$ and the parameter space of polarized hyperplane sections}
Since no sub-scheme of $B$ with Hilbert polynomial $3t+1$ is a planar one, then we have a natural injection 
$\sH^{B}_{3}\rightarrow\mathbb G(\mP^{2},{\check{\mP^{6}}})$ given 
by the map $[\Gamma]\mapsto[{\rm{Ann}}\langle \Gamma\rangle]$ where 
${\rm{Ann}}\langle \Gamma\rangle\subset{\check{\mP^{6}}}$ is the web 
of hyperplanes containing the $3$-dimensional projective subspace of $\mP^6$ spanned 
by $\Gamma$. Let $\tau\colon \sR\to\sH^{B}_{3}$ be the pull-back over 
$\sH^{B}_{3}$ of the universal family $\sU^{3}_{7}$ of $G(\mP^{2},{\check{\mP^{6}}})$. By Lemma \ref{lines in V} we know that $\sH^B_3$ is isomorphic to $G(2,5)$. We are going to decide the $\mP^2$-bundle $\sR\to G(2,V)=\sH^B_3$.  Note that $\sR$ is smooth and irreducible.
\begin{prop}\label{coordinatefree}
$\sR$ is the flag variety $F(2,3,V)$, namely, it parameterizes the configuration
$line \subset plane \subset \mP(V)$. In particular $\sR=\mP(\sG^*)$, where $\sG$ is the dual of the universal subbundle of rank three on $G(3,V)$.
\end{prop}
\begin{proof} We recall that $G(2,5)=\sH^B_3$ since Lemma \ref{lines in V}. For any $[l]\in G(2,5)$ we denote by $[{\rm{Ann}}\langle l\rangle]\subset{\check{\mP^{6}}}$ the plane parameterising the hyperplanes of $\mP(V)$ containing $l$. It is easy to see that  the isomorphism $G(2,5)\to \sH^B_3$ of Lemma \ref{lines in V}  given by $[l] \mapsto [\Gamma]:=[\pi_{B\star}\pi_{\sF}^{\star}(l)]$ gives an injective morphism $F(2,3,V)\to \sR$.
 Since $\sR$ and $F(2,3,V)$ are smooth irreducible and of the same dimension it follows that $F(2,3,V)\to \sR$ is an isomorphism.
\end{proof}

\subsubsection{Explicit presentation of a compactification of $\mathring \sH_{\rm{{c.t.}}}$}
Finally we consider the space $\mP(S^2 \sG^*)$ which is the Hilbert scheme of conics in $\mP(V)$. The natural morphism 
$\rho_1\colon \mP(S^2 \sG^*)\to G(3,V)$ clearly expresses $\mP(S^2 \sG^*)$ as a $\mP^5$-bundle over $G(3,V)$. We consider $ {\mathring{\rho_1}}\colon {\mathring{ \mP}} ( S^2 \sG^{*} )\to {\mathring{\mathbb G}}$ its restriction over ${\mathring {\mathbb G}}$, see Equation (\ref{openopen}), and where ${\mathring{ \mP}} ( S^2 \sG^{*} )$ parameterises only the smooth conics of $\mP(V)$.

\begin{prop}\label{cleaning} There is a natural $G$-equivariant isomorphism between $\mathring \sH_{\rm{{c.t.}}}$ and ${\mathring{ \mP}} ( S^2 \sG^{*} )$. Moreover $\mathring{\mathbb G}$ is isomorphic to  $\mathring \sT$.
\end{prop}
\begin{proof} We use notations of Lemma \ref{castelnuovo} and of Lemma \ref{lines in V}.

Consider the diagram \ref{tris&line}. By definition of $\sH_{\rm{{c.t.}}}$ the morphism ${\mathring{\delta}}\colon{\mathring{ \mP}} ( S^2 \sG^{*} )\to \mathring \sH\hookrightarrow \sH_{\rm{{c.t.}}}$ given by $[q] \mapsto [R]:=[\pi_{B\star}\pi_{\sF}^{\star}(q)]$ admits an inverse over $\mathring \sH_{\rm{{c.t.}}}$. This implies that $\mathring \sH_{\rm{{c.t.}}}$ is irreducible. Hence $\mathring \sT$ is irreducible.  We recall the morphism $\alpha\colon\mathring{\mathbb G} \to \check\mP^6$ given in Equation (\ref{alpha}) and we compose it by ${\mathring{\rho_1}}\colon {\mathring{ \mP}} ( S^2 \sG^{*} )\to {\mathring{\mathbb G}}$. By our construction the Stein factorisation $\mathring \sigma\circ\mathring\rho$ of $\mathring \epsilon\colon \mathring \sH_{\rm{{c.t.}}}\to\check\mP^{6}$ and the Stein factorisation of 
$\mathring\alpha\circ{\mathring{\rho_1}}\colon {\mathring{ \mP}} ( S^2 \sG^{*} )\to\check\mP^{6}$ of the morphism ${\mathring{ \mP}} ( S^2 \sG^{*} )\to  \check\mP^6$ are $G$-equivariant compatible with ${\mathring{\delta}}$ and with respective $G$-universal properties. Hence $\mathring{\mathbb G}$ is $G$-isomorphic to  $\mathring \sT$.
\end{proof}

\begin{rem} \label{globalmorphisms} By Lemma \ref{coordinatefree} and by Proposition \ref{cleaning} we have that the Stein factorisation of the morphism $\sR\to \check\mP^{6}$ is given by $\mu\colon\sR\to\sT$ followed by $\sigma\colon \sT\to \check\mP^{6}$ where by \ref{coordinatefree} it holds that $\sT=G(3,V)$. 
\end{rem}

\subsection{Group action on rational sextics of conic type }
 We have seen above that ${\mathring {\sH}}\simeq_G{\mathring{\mP}}(S^2 \sG^*)$, $\sR=F(2,3,5)$ and $\sT=\mathbb G (3,V)$. Moreover 
 $\sH^{B}_{3}\simeq_G \mathbb G (2,V)$ by Lemma \ref{lines in V}.
 
 Then to understand the birational nature of the $G$-action on the following diagram:
     \begin{equation}\label{sumupvan1}
\xymatrix{& {{\mathring{\sH}}_{\rm{{c.t.}}}}\ar[d]_{{\mathring{\rho}}}\\
{\mathring{\sR}}\ar[d]_{\tau}\ar[r]^{\mu}&  {\mathring{\sT}}
\ar[r]^{\sigma} & \check\mP^{6}\\
\sH^{B}_{3} } 
\end{equation}
\noindent we can study the natural $G$-action on open orbits of the following one:

 \begin{equation}\label{sumupvan}
\xymatrix{& \mP(S^2 \sG^*)\ar[d]_{\rho}\\
F(2,3,V)\ar[d]_{\tau}\ar[r]^{\mu}&  G(3,V)
\ar[r]^{\sigma} & \check\mP^{6}\\
G(2,V)} 
\end{equation}
\noindent since these two diagrams are $G$-birational.

\begin{thm}\label{firstrational}
 ${\mathring{\sH}}_{\rm{{c.t.}}}//G$ is rational. 
\end{thm}
\noindent{\it{First proof.}} The claims follows easily by Lemma \ref{quotient grassmannian} and by the Lemme de descente: c.f. \cite[Th\'eor\`eme 2.3]{DN}.

\section{Genus 4 spin curve with a vanishing theta-null}
Let $\overline{{\rm{S}}^+_4}$ be the moduli space of genus $4$ spin curves. By \cite{cornalba} we know that it is a projective variety. The forgetful morphism $\xi\colon \overline{{\rm{S}}^+_4}\to\overline{\sM_4}$ exibits it as a $136$-to-$1$ cover of the Deligne-Mumford compactification of the moduli space of smooth genus $4$ curves. Inside $\overline{\sM_4}$ 
there is the divisor $\overline{\sM^{{\rm{null}}}_4}$ which is the closure of the loci $\sM^{{ \rm{null},0  }}_4\subset \sM^{{\rm{null}}}_4$ which parameterizes the genus-$4$ smooth curves $C$ whose canonical model is a transversal intersection inside $\mP^3$ of a quadric cone and a smooth cubic surface and such that ${\rm{Aut}}(C)={\rm{id_{C}}}$. We set  ${\rm{S}}^{{\rm{null}}}_{4}$  to be the  $\xi$-pull-back of $\sM^{{\rm{null}}}_4$, hence ${{{\rm{S}}^{{\rm{null}}}_{g}}}=
{{{\rm{S}}^{{\rm{null}},0 }_{4}}}\sqcup \Theta_{g,\rm{null}}$; see Equation (\ref{dsunion}) of the Introduction. We want to study the rationality problem of ${{{\rm{S}}^{{\rm{null}},0 }_{4}}}$.
 
\begin{prop}\label{reducedness} The divisor ${{{\rm{S}}^{{\rm{null}},0}_{4}}}$ is reduced.
\end{prop}
\begin{proof}
It is well-known that $\sM^{{\rm{null}}}_{4}$ is reduced irreducible and that  the restriction ${\rm{S}}^{{\rm{null}},0}_4 \to \sM^{{\rm{null}}}_{4}$ of the forgetful morphism is \'etale. Hence the general point of ${\rm{S}}^{{\rm{null}},0}_4$ is smooth. Hence ${{{\rm{S}}^{{\rm{null}},0}_{4}}}$ is reduced.
\end{proof}

\subsection{The Prym canonical map}
\label{Prymcanonical}

Let $[C,\theta]\in {\rm{S}}^{+{\rm{null}}}_{4}$ be a general element, in particular the automorphism group of $C$ is the trivial one. Let $\delta$ be the unique $g^1_3$ on $C$. The map $\varphi_{|\theta+\delta|}\colon C\to \mP(H^0(C,\sO_C(\theta+\delta ))^\star)$ is known as the Prym canonical map and it is known that $\theta-\delta$ is a 2-torsion divisor. 

\subsubsection{The Prym canonical map is a morphism}
Actually $\varphi_{|\theta+\delta|}$ can be geometrically interpreted thanks to our theory. Indeed we can recover it via the  restriction to $\varphi^{-1}(R)$ of the diagram (\ref{universal&projection}) where $[R]\in\sH$. First we show:

\begin{lem} \label{image}For a general $[C,\theta]\in  {\rm{S}}^{ {\rm{null}}, 0}_4$ the 
    linear system $|\delta+\theta|$ gives a morphism 
    $\phi_{|\delta+\theta|}\colon C\to\mP^{2}$.
\end{lem}
\begin{proof} Since $\theta$ is ineffective, $h^{0}(C,\sO_{C}(\theta+\delta))=3$. If $|\delta+\theta|$ had a base point $p$ then $h^{0}(C,\sO_{C}(\theta+\delta-p))=3$. This would imply $h^{0}(C,\sO_{C}(\theta))\geq 1$; a contradiction.
\end{proof}

\subsubsection{Thetasymmetric curves}
We need to 
understand some geometry of the morphism $\phi_{|\delta+\theta|}\colon 
C\to\mP^{2}$. Our argument needs  the following Proposition which has its own interest:

\begin{prop}\label{symmetry}
Let $\Gamma$ be a smooth non-hyperelliptic curve of genus $4$ 
with two different $g^1_3$'s $\delta$ and $\delta'$.
Let $\theta$ be an ineffective theta characteristic.
Then 
$|\theta+\delta|$ and $|\theta+\delta'|$ are base point free
and the images $M$ and $M'$ of $\Gamma$ defined, respectively,
by these linear systems are
plane sextic curves.
Moreover 
$h^0(\Gamma, \sO_{\Gamma}(\theta+\delta-\delta'))>0$, which is equivalent to
$h^0(\Gamma, \sO_{\Gamma}(\theta+\delta'-\delta))>0$,
if and only if
$M$ and $M'$ have triple points.
If this condition is satisfied, then
$h^0(\Gamma, \sO_{\Gamma}(\theta+\delta-\delta'))=h^0(\Gamma, \sO_{\Gamma}(\theta+\delta'-\delta))=1$ and
$M$ and $M'$ have a unique ordinary triple point respectively.
\end{prop}

\begin{proof}
By the symmetry of $\delta$ and $\delta'$ 
it suffices to prove the assertions for
$|\theta+\delta|$ and $M$. It is easy to see that $h^0(\Gamma, \sO_{\Gamma}(\theta+\delta))=3$. Since $\Gamma$ is not hyperelliptic,
$M$ has no quadruple point.
  Assume that $h^0(\Gamma, \sO_{\Gamma}(\theta+\delta-\delta'))>0$.
  Let $\eta \in |\theta+\delta-\delta'|$. 
  We have $\eta+\delta'\sim \pi^*\sO_M(1)$.
  Since $\deg \eta=3$ and $\dim |\delta'|=1$, 
  the divisor $\eta$ consists of three points
  which are the pull-back of a triple point.  
  By reversing the argument, we see that
  the divisor consisting of the pull-back of a triple point 
  is a member of $|\theta+\delta-\delta'|$.
  Since $M$ has only a finite number of triple points,
  $M$ has actually a unique triple point and
  $h^0(\Gamma, \sO_{\Gamma}(\theta+\delta-\delta'))=1$.  
\end{proof}

We think that genus $4$ curves as those of Proposition \ref{symmetry} deserve a name:
\begin{defn} Let $\delta$ and $\delta'$ be two distinct $g^1_3$ on a curve $C$ of genus $4$. 
The spin curve $[C,\theta]$ is called {\it{$g^1_3$-thetasymmetric}} if  $h^0(C,\sO_C(\theta+\delta-\delta'))=1$.
\end{defn}

\begin{rem} Actually $g^1_3$-thetasymmetric curves correspond to the rational sextics of Lemma \ref{threecases} $({\rm{II}})$. \end{rem}

\subsubsection{The image of the Prym canonical map}
Now we can describe the image $M$ of $C$ by the morphism $\phi_{|\delta+\theta|}\colon C\to\mP^2$ where  $[C,\theta]\in { \rm{S}}^{{\rm{null}},0}_4$; more precisely:

\begin{prop}\label{threesexstic} Let $[C,\theta]\in  {\rm{S}}^{{\rm{null}},0}_4$ be a general element.
    Then the image $M$ of the morphism 
    $\phi_{|\theta+\delta|}\colon C\to\mP^{2}$ is a sextic with six 
    nodes which are given by the six point of the $(4,6)$ configuration associated to $4$ lines $L_1, L_2, L_3, L_4$ in general position.
 \end{prop}
\begin{proof} By Lemma \ref{image} if ${\rm{deg}}(M)<6$ then ${\rm{deg}}(M)=2$ or $3$ since ${\rm{deg}}(\sO_C(\theta+\delta)=6$. We distinguish these two cases.

If ${\rm{deg}}(M)=2$ then $\theta+\delta=2\delta$, a contradiction. 
    
    If 
    ${\rm{deg}}(M)=3$ and $M$ is singular then $C$ is hyperelliptic, 
    a contradiction to the generality of $[C,\theta]\in \sS^{{\rm{null}},0}_4$. 
    
    If  ${\rm{deg}}(M)=3$ and $M$ is smooth, then $C$ is 
    bielliptic and again we exclude this case by the generality of 
    $[C,\theta]$ 
    since the bielliptic locus is $6$-dimensional.

    We have shown that ${\rm{deg}}(M)=6$. Since  $C$ is not hyperelliptic then the maximal order of a singular 
    point of $M$ is $3$. On the other hand if $M$ has a singular 
    point of multiplicity $3$ then by Proposition \ref{symmetry}, $C$ 
    is $g^{1}_{3}$-symmetric, a contradiction. Hence $M$ is a sextic and its singular points are 
    nodes. Since $g(C)=4$ and the arithmetical genus $\rho_{a}(M)=10$ 
    then we have that $M$ has exactly $6$ points of multiplicity $2$. Now we show that 
    these six points $n_{ij}$ where $i,j=1,2,3,4$, $i< j$ are 
    in a special position. We denote by $\mathfrak T$ the set of the transpositions $(i,j)$ where $4\geq j>i\geq 1$.
    
     Let
$\sigma\colon S\to\mP^{2}$ be the blow-up along the 
    six points $n_{ij}$, $(i,j)\in \mathfrak T$.     
    Then 
    $$h_{\mid C}\sim \theta+\delta$$ where $h$ is the total
    $\sigma$-trasform of a line. Set 
    $E_{ij}:=\sigma^{-1}(n_{ij})$ and 
    learly $C\in |6h-2\sum_{(i,j)\in \mathfrak T} E_{ij} | $. Since there is a unique 
    $g^{1}_{3}$, by adjunction on $S$ it holds that 
    $2\delta\sim\omega_{C}\sim
    (3h-\sum_{(i,j)\in\mathfrak T} E_{ij})_{\mid C}$. 
    Set 
    $$\alpha:=(\sum_{(i,j)\in \mathfrak T} E_{ij})_{|C}.$$
    \noindent
    Since $h_{\mid C}\sim \theta+\delta$ it holds that $2\delta\sim 
    3\theta+3\delta-\alpha$. Then $\alpha\sim 
    2\theta+\theta+\delta\sim \omega_{C}+h_{\mid C}=4h_{\mid C}-\alpha$.
    
    \noindent
    {\it{Claim: $H^{0}(S,\sO_{S}(4h-\sum_{(i,j)\in \mathfrak T} E_{ij}))=H^{0}(C,\sO_{C}(4h-\sum_{(i,j)\in \mathfrak T}E_{ij}))$.}} Indeed since $\alpha\simeq (4h-\sum_{(i,j)\in\mathfrak T} E_{ij})_{\mid C}$ then by the exact sequence 
    giving the sheaf $\sO_{C}$ as an $\sO_{S}$-module we obtain the 
    following exact sequence:
    \begin{equation}\label{seq for two}
	0\to\sO_{S}(-2h+\sum_{(i,j)\in\mathfrak T}E_{ij})\to \sO_{S}(4h-\sum_{(i,j)\in\mathfrak T}E_{ij})\to\sO_{C}(\alpha)\to 0
\end{equation}
Since $H^{0}(S,\sO_{S}(-2h+\sum_{(i,j)\in\mathfrak T}E_{ij}))=0$ and since by Serre's duality 
    $H^{2}(S,\sO_{S}(-2h+\sum_{(i,j)\in\mathfrak T} E_{ij}))=0$ then Rieman-Roch's theorem on $S$ gives 
    $-H^{1}(S,\sO_{S}(-2h+\sum_{(i,j)\in\mathfrak T} E_{ij}))=\chi(\sO_{S}(-2h+\sum_{(i,j)\in\mathfrak T} E_{ij}))=0$. This immediately shows our claim.
    
    Finally we have only to point out that the divisor $\sum_{(i,j)\in\mathfrak T} E_{ij}$ itself cuts the divisor $\alpha$ on $C$.  By the isomorphism $H^0(S, \sO_S(4h-\sum_{(i,j)\in\mathfrak T} E_{ij}))=H^0(C,\sO_C(\alpha))$ it follows that there exists a unique $\varsigma\in H^0(S, \sO_S(4h-\sum_{(i,j)\in\mathfrak T} E_{ij}))$ such that $\varsigma'_{|C}=\alpha$. By tensorialising the sequence (\ref{seq for two}) by $\sO_S(-\sum_{(i,j)\in\mathfrak T} E_{ij})$
    this means that $h^{0}(S,\sO_{S}(4h-2\sum_{(i,j)\in \mathfrak T}E_{ij}))=1$, since $h^0(S,\sO_{S}(-2h))= h^1(S,\sO_{S}(-2h))=0$. This shows that the six points are 
    three by three on 
    four lines since $M$ is irreducible.
\end{proof}

\begin{lem}\label{threesexsticbis}
Let $[C,\theta]\in { \rm{S}}^{{\rm{null}},0}_4$ be a general element and let $M$ be the image of the Prym canonical morphism 
    $\phi_{|\theta+\delta|}\colon C\to\mP^{2}$.
    Let $\sigma\colon S\to\mP^{2}$ be the 
    blow-up of $\mP^{2}$ at the six 
    nodes $n_{ij}\in \mP^{2}$ of $M$, let 
    $E_{ij}:=\sigma^{-1}(n_{ij})$ and let 
    $E_{ij|C}:=a_{ij}+b_{ij}$, where $(i,j)\in \mathfrak T$. Let 
    $a_{ij}+a^{1}_{ij}+a^{2}_{ij}$ and $b_{ij}+b^{1}_{ij}+b^{2}_{ij}$ be the two 
    distinct effective divisors of the unique trigonal series on $C$ which 
    contains $a_{ij}$ and respectively $b_{ij}$. Then 
    $\{\sigma(a^{1}_{ij}),\sigma(a^{2}_{ij}),\sigma(b^{1}_{ij}), \sigma(b^{2}_{ij}) \}$ is a 
    set of collinear points on a line $L_{ij}$ passing through the point 
    $n_{rs}$ where $\{i,j\}\cap\{r,s\}=\emptyset$ and $\{i,j,r,s\}=\{1,2,3,4\}$.
 \end{lem}
 \begin{proof} The proof is quite easy.  We need to show that $h^0(C,\sO_C(h_{|C}- a^{1}_{ij}-a^{2}_{ij}-b^{1}_{ij}-b^{2}_{ij}))>0$. Indeed let $a+b+c+d$ be the 
     unique effective divisor contained inside $|\theta+a_{ij}|$ and 
     let $a'+b'+c'+d'$ be the 
     unique effective divisor contained inside $|\theta+b_{ij}|$. 
     Since $2\theta\sim K_{C}$ then 
     $$a+b+c+d+a'+b'+c'+d'\sim 
     K_{C}+ E_{ij|C}.
     $$\noindent 
     On the other hand $h_{|C}\sim 
     \theta+\delta$ where $h$ is the pull-back of the line of 
     $\mP^{2}$. Then $a+b+c+d=h_{|C}-a^{1}_{ij}-a^{2}_{ij}$ and  
     $a'+b'+c'+d'=h_{| C}-b^{1}_{ij}-b^{2}_{ij}$; that is:
     \begin{equation}\label{description}
      K_{C}+ E_{ij|C}\sim 2h_{|C}-a^{1}_{ij}-a^{2}_{ij}-b^{1}_{ij}-b^{2}_{ij}.
     \end{equation} Finally by adjuntion 
     $K_{C}\sim 3h_{|C}-(\sum_{(l,m)\in\mathfrak T} E_{lm})_{|C}=3h_{|C}-\alpha$.
     Then by equation (\ref{description}) we have:
     $$3h_{|C}-\alpha+E_{ij|C}\sim 2h_{|C}- a^{1}_{ij}-a^{2}_{ij}-b^{1}_{ij}-b^{2}_{ij}
     .$$\noindent
     Now we cancel $h_{|C}$ in both members of the above equality to obtain:
    $$2h_{|C}-\alpha+E_{ij|C}=2h_{|C}-\sum_{(lm)\neq (ij)}E_{lm|C}\sim
    h_{|C}- a^{1}_{ij}-a^{2}_{ij}-b^{1}_{ij}-b^{2}_{ij}$$ 
    By Lemma \ref{threesexstic} the first claim follows since it there exists a (unique)
    conic through $5$ points. Let us call $T_{34}$ the unique line such that $\sigma^{\star}(T_{34})$ contains $a^{1}_{12},a^{2}_{12},b^{1}_{12},b^{2}_{12}$. We show that $n_{34}\in T_{34}$. 
    Indeed the effective divisor 
    $$
    2h_{|C}-\alpha +E_{12|C}=h_{|C}-(E_{13}+E_{23}+E_{34})_{|C}+h_{|C}-(E_{14}+E_{24}+E_{34})_{|C} +E_{34|C}. 
    $$
    Since $h_{|C}-(E_{13}+E_{23}+E_{34})_{|C}$ and $h_{|C}-(E_{14}+E_{24}+E_{34})_{|C}$ are disjoint from $C$ it follows that
    $$
    E_{34|C}=h_{|C}- a^{1}_{12}-a^{2}_{12}-b^{1}_{12}-b^{2}_{12}.
    $$\noindent
    The remaining 5 cases can be treated identically.
   \end{proof}

\section{Construction of genus 4 spin curve with a vanishing theta-null via  sextics of conic type}
We restrict Diagram (\ref{universal&projection}) to a general rational sextic of conic type $R$:
\begin{equation}
\label{universal&projectionsecond}
\xymatrix{ & C(R)\subset \sU_1 \ar[dl]_{{ \pi}}
\ar[dr]^{\varphi}\\
M(R)\subset \sH^B_1=\mP^{2} &  & R\subset Z\subset  B\subset \mP^{6} }
\end{equation}
where $C(R):=\varphi^{-1}(R)$ and $M(R):= \pi(C(R))$.  

\subsubsection{The curve of genus $4$}
By Lemma \ref{threecases} we know that we can realise $R\subset B$ as a curve $R\subset Z=\langle R\rangle\cap B$ where on $Z$ there is a polarisation $|\lambda|$ such that $R\in |2\lambda|$ and $\phi_{|\lambda|}\colon Z={\rm{Bl}}_{a_{1},a_{2},a_{3},a_{4} }\mP^{2}\to\mP^2$. We set $e_i:= \phi_{|\lambda|}^{-1}(a_i)$, $i=1,...,4$.

\begin{lem}\label{primaCIII} Let $[R]\in\sH_{\rm{{c.t.}}}$ be a general element. 
The scheme $C(R)$ is a smooth curve of genus $4$ contained inside the 
universal family of lines of $B$. 
\end{lem}
\begin{proof} By Proposition \ref{prop:FN} $(1)$ and by Proposition \ref{prop:Cd1III} it follows that $C(R)$ is smooth and the ramification divisor of $\varphi_{|C(R)}\colon C(R)\to R\subset B$ is simple. Since $B_{\varphi}\in |-K_{B}|$ and $6={\rm{deg\,{R}}}$,
we obtain $g(C(R))=4$ by the Hurwitz's formula.
\end{proof}

\subsubsection{The geometry of the plane model} Now we study the morphism $\pi_{|C(R)}\colon C(R)\to M(R)$.

\begin{lem}
    \label{orabisecantiIII}  Let $[R]\in\sH_{\rm{{c.t.}}}$ be a general element. Then $M(R)$ is an irreducible sextic with $6$ nodes. In particular 
$R$ has exactly $6$ bi-secant lines on $B$.
\end{lem} 
\begin{proof} By standard geometry of smooth del Pezzo surface of degree $5$ the lines $\beta_{ij}\in |\lambda- e_i-e_j|$ where $1\leq i<j\leq 4$ are the six bisecants of $R$. By Proposition \ref{prop:FN} $(4)$ it follows that $M(R)$ is a plane sextic. Indeed $L_{[l]}\cdot M(R)=6$ where $L_{[l]}\subset\sH^B_1=\mP^2$ is the locus which parameterizes those lines $m\subset l$ such that $m\cap l\neq \emptyset$. Hence $p_a(M)=10$. By generality of $R$ the curve $M(R)$ does not have any other singular point except those $6$ nodes due to the $6$ bisecants of $R$. Hence $\pi_{|C(R)}\colon C(R)\to M(R)$ is the normalisation morphism since Lemma \ref{primaCIII}.
\end{proof}

\subsubsection{The associated thetacharacteristic}
Now we show that our method gives an interpretation of the Prym canonical map. In order to help the reader we recall and we fix notation.
\begin{nota}\label{useesu} If $l\subset B$ is a line of $B$ we denote by $[l]$ its corresponding point inside $\sH^{B}_{1}$. We denote by $L_{[l]}$ the line inside $\sH^B_1$ parameterising the lines of $B$ which intersect $l$.
We have set $Z:={\rm{Bl}}_{a_{1},a_{2},a_{3},a_{4} }\mP^{2}$. In particular $Z$ can be identified to its image $\langle R\rangle \cap B$. As always we denote by $[\epsilon_{i}]\in
\sH^{B}_{1}$ the point which correspond to the line $\epsilon_i:=\mu^{-1}(a_{i})\subset Z$, where $i=1,2,3,4$ and $\phi_{|\lambda|}\colon Z\to\mP^2$ is the blow-up. We also denote by $[\beta_{ij}] \in \sH^{B}_{1}$ the point which corresponds to the unique element $\beta_{ij}$ of $|\lambda-\epsilon_i-\epsilon_j|$; in the sequel instead of writing $1\leq i<j\leq 4$ we often write $(i,j)\in \mathfrak T$, where $(i,j)$ is the corresponding transposition. \end{nota}
\begin{prop}\label{newIII} Let $[R]\in\sH_{\rm{{c.t.}}}$ be a general element. The morphism
$$\pi_{|C(R)}\colon C(R)\rightarrow M(R)\subset\sH^{B}_{1}\simeq\mP^{2}$$ is
induced by the linear system $|\delta+\theta(R)|$ where $\theta(R)$ is an 
ineffective theta characteristic over $C(R)$. Moreover $[C(R)]\in \sM^{{\rm{null}}}_g$, and $[(C(R),\theta(R)]\in {{{\rm{S}}^{{\rm{null}},0 }_{g}}}$.
\end{prop} 
\begin{proof}  Set $C=C(R)$ and $M=M(R)$. The six nodes of $M$ are in special position since $[\beta_{ij}]\in L_{\epsilon}\cap L_{[\epsilon_j]}$, $(i,j)\in \mathfrak T$. It easily follows that if $\sigma\colon S\to \sH^{B}_{1}=\mP^2$ is the blow-up at the six points $[\beta_{ij}] \in \sH^{B}_{1}$, $E_{ij}:=\sigma^{-1}([\beta_{ij}])$, $(i,j)\in \mathfrak T$ and $|h|$ is the linear system giving $\sigma\colon S\to \sH^{B}_{1}$ then $C\in |6h-2\sum_{(i,j)\in \mathfrak T}E_{ij}|$ and $h^0(C,\sO_C(h_{|C}))=3$. We can identify $( \sH^{B}_{1}, \{[\epsilon_1],[\epsilon_2],[\epsilon_3],[\epsilon_4] \}$ to $\mP^2$ with the standard projective frame.

To get information on $S$, we now switch to the polarised hyperplane section $(Z,|\lambda|)$ which contains $R$ . Let $\delta$ be the $g^1_3$ on $C$ given by the $3$-to-$1$ covering $\varphi_{|C|}\colon C\to R\subset B$. Since on $Z$ it holds that $(\lambda-\epsilon_1-\epsilon_2)\cdot (\lambda-\epsilon_3-\epsilon_4)=1$ it easily follows that the line $L_{[\beta_{12}]}\subset \sH^{B}_{1}$ parameterising the lines of $B$ which intersect $\beta_{12}\subset B$ passes through the node of $M$ supported on $[\beta_{34}]$ and other four points $p,q,r,s$. We set $C\cap E_{12}=\{ a_{12}, b_{12}\} \subset S$. Now we stress that the forgetful morphism $\pi_{|C}\colon C\to  \sH^{B}_{1}$ coincides to $\sigma_{|C}$. Hence by the geometry of $B$ it follows that 
$$2\delta\sim a_{12}+ b_{12}+p+q+r+s=(h-E_{34}+E_{12})_{|C}$$
and w.l.o.g. we can set  $a_{12}+ p+q\sim b_{12}+r+s\sim \delta$.

The generality of $R$ implies that the hyperplane section $Z$ is a general one. Then  it follows that $h^0(C,\sO_C(h_{|C}-\delta))=0$ since the images by $\sigma\colon S\to\sH^B_1$ of $a_{12},p,q$ are not collinear. Finally we claim that 
$$\theta:=h_{|C}-\delta
$$ is an ineffective theta characteristic. By adjunction $K_C\sim (3h-\sum_{(i,j)\in \mathfrak T}E_{ij})_{|C}$ and by generality $C$ does not intersect the strict trasform of any of the lines containing three nodes of $M$. This implies that
$$
(3h-\sum_{(i,j)\in \mathfrak T}E_{ij})_{|C}\sim (2h-E_{24}-E_{34}-E_{23})_{|C}\sim(h+E_{12}-E_{34})_{|C}
$$
We set $\theta(R):=\theta$ and the claim follows.
\end{proof}
\section{Reconstruction of  sextics of conic type via genus-4 spin curves with a vanishing theta-null} We are ready to show the main technical result of this paper. The idea is to use $\phi_{|\delta+\theta|}\colon C \to \mP^2=\mP(H^0(C,\sO_C(\delta+\theta)^{\vee})$ to identify $\mP^2$ and a suitable conic to the couple $(\sH^B_1,\Omega)$ in order to force data coming from the Prym canonical map to produce $[R]\in{\mathring{\sH}}_{{\rm{c.t.}}}$ such that $[(C,\theta)]=[(C(R), \theta(R))]$.

 Let $[C,\theta]\in {\rm{S}}^{{\rm{null,0}}}_4$ be a general element and let $M\subset \mP^2$ be the sextic with six nodes given by Proposition \ref{threesexstic}. Denote by $L_1, L_2, L_3, L_4$ the four lines which give the $(4,6)$ configuration, see also Lemma \ref{threesexsticbis}. We set $$\{n_{ij}\}\in L_i\cap L_{j}$$ where $(i,j)\in \mathfrak T$. Let $\sigma\colon S\to \mP^2$ be the blow-up at the points  $n_{ij}$ and set $E_{ij}:=\sigma^{-1}(n_{ij})$ where $(i,j)\in \mathfrak T$. We notice that $C\hookrightarrow S$ and we set
$$
E_{ij|C}:= a_{ij}+b_{ij},\,\,  (i,j)\in \mathfrak T.
$$
\noindent
If $\delta$ denotes the unique $g^1_3$ over $C$ we also set $\delta\sim a_{ij}+a^{1}_{ij}+a^{2}_{ij}$, $\delta\sim b_{ij}+b^{1}_{ij}+b^{2}_{ij}$ $(i,j)\in \mathfrak T$. We set 
$$
M:=\sigma(C).
$$


\begin{thm}\label{perlomodulothree}{\bf{(Reconstruction theorem)}} Let $[C,\theta]\in {\rm{S}}^{{\rm{null}},0}_4$ be a general element. Then there exists $[R]\in \mathring \sH_{{\rm{c.t.}}}$ such that  $[C,\theta]=\pi_{\sS^{+}_{4}}([R])$.
\end{thm}
\begin{proof} Consider the blow-up $\sigma\colon S\to \mP^2$ which induces the Prym canonical map on $C$. By Lemma \ref{threesexsticbis} there exists a line $L_{ij}\subset \mP^2$ such that $n_{rs}\in L_{ij}$ and $$\sigma(a^{1}_{ij}),\sigma(a^{2}_{ij}),\sigma(b^{1}_{ij}), \sigma (b^{2}_{ij})\in L_{ij}$$ where $(i,j),(r,s)\in \mathfrak T$, $\{i,j\}\cap\{r,s\}=\emptyset$.

Now we consider the four lines $L_1,...,L_4\subset\mP^2$ of Proposition \ref{threesexstic}. We call $\{L_1,...,L_4\}$ {\it{the configuration of lines associated to $M$}}. For any  smooth conic $Q\in \mP(H^0(\mP^2, \sO_{\mP^2}(2)))$ we set:

\begin{equation}\label{smooth}
e_{t}(Q):={\rm{Pol}}_{Q}(L_{t}),\,\, t=1,2,3,4,
\end{equation}\noindent
where ${\rm{Pol}}_{Q}(...)\colon \mP^{2\vee}\to \mP^2$ is the standard polarisation morphism induced by $Q$. By Proposition \ref{prop:polar} we can define an identification
 $(\mP^2, Q)\leftrightarrow_{Q}( \sH^{B}_{1},\Omega)$ where given a point $x\in\mP^2$ we write: $x\leftrightarrow_{Q} [l_x]$, if $l_x\subset B$ is the line inside $B$ corresponding to the point $[l_x]\in\sH^{B}_{1} $. In particular letting
 $$\mP^2 \ni e_{t}(Q) \leftrightarrow_{Q} [\epsilon_{t}]\in\sH^{B}_{1}$$
 we can identify a unique line $\epsilon_t\subset B$,  where $t=1,2,3,4$.

 We set:
 $$
 n_{ij} \leftrightarrow_{Q} [\beta_{ij}]\in\sH^{B}_{1}
 $$
 where $(i,j)\in \mathfrak T$.
 By the meaning of the polar geometry associated $( \sH^{B}_{1},\Omega)$ given in Proposition \ref{prop:polar} we can transfer intersection conditions on $B$ to polarity relations on $\mP^2$. Hence it holds
 $$\epsilon_i\cap \beta_{ij}\neq\emptyset$$
 where $i=1,...,4$, $(i,j)\in \mathfrak T$, since $L_i$ corresponds to the line 
 $L_{[\epsilon_{i}]}$ in $\sH^{B}_{1}$ which parameterises those lines $l\subset B$ such that $l\cap \epsilon_{i}\neq \emptyset$ where $i=1,2,3,4$. Note that $Q$ must be a smooth conic in order to satisfy Equation (\ref{smooth}).

 Since the configuration of lines associated to $M$ is given by four lines in general position and $e_i(Q)\not\in L_j$, that is $[\epsilon_{i}] \not\in L_{[\epsilon_{j}]}$, on $B$ it holds that $\epsilon_i \cap \epsilon_j=\emptyset$ if $i\neq j$, $1\leq i,j\leq 4$. Hence the linear span $H:=\langle\epsilon_1,\epsilon_2,\epsilon_3\rangle$ is a $\mP^5$ inside the $\mP^6$ which contains $B$.
Since $n_{23}\in L_2\cap L_3$ it holds that $\beta_{23}\cap \epsilon_2\neq\emptyset$ and $\beta_{23}\cap \epsilon_3\neq\emptyset$. Hence $\beta_{23}\subset H$. By the same argument we have that $\beta_{12},\beta_{13}\subset H$. Let
$$
\mP_{14\perp 23}:=\{Q\in \mP(H^0(\mP^2, \sO_{\mP^2}(2)))\mid 0={{Q}}(n_{14},n_{23})\}
$$
be the hyperplane given by those conics such that $n_{14}\perp_{Q} n_{23}$. It is obvious that ${\rm{dim}}_{\mathbb C}\mP_{14\perp 23}=4$. By definition it follows that for any smooth conic $Q\in \mP_{14\perp 23}$, in the identification $(\mP^2, Q)\leftrightarrow_{Q}( \sH^{B}_{1},\Omega)$ it holds that 
$$
\beta_{14}\cap\beta_{23}\neq\emptyset.
$$
Since $\beta_{14}\in L_1$ and $\beta_{23}\not\in L_1$ it holds that $\beta_{14}\cap\epsilon_1\neq \beta_{14}\cap\beta_{23}$. Since $\epsilon_1,\beta_{23}\subset H$ we have that the line $\beta_{14}\subset H$. There are only two cases: $\epsilon_4\subset H$ or $\epsilon_4\not\subset H$. Assume that $\epsilon_4\not\subset H$. Now by generality of $\epsilon_{i}$, $i=1,2,3,4$ it holds that $Z:=H\cap B$ is a smooth hyperplane section. In particular there exists another line $\epsilon_5\subset Z$ such that $\tau\colon Z\to\mP^2$ is the blow-up at the four points $x_1,x_2,x_3, x_5$ and $\epsilon_i=\tau^{-1}(x_i)$, $i=1,2,3,5$. Inside $Z$ there are $10$ lines and among them $8$ are known:
$\epsilon_1,\epsilon_2,\epsilon_3,\epsilon_5,\beta_{12},\beta_{13},\beta_{23}, \beta_{14}$. By construction there are only three lines inside $Z$ which intersects $\epsilon_1$ and they are 
$\beta_{12},\beta_{13}, \beta_{14}$. Hence it holds that $\beta_{14}=\beta_{15}$ where $\beta_{ij}$ is the $\tau$-strict transform of the line through $x_i$ and $x_j$ where $i\neq j$, $i,j=1,2,3,5$.
We turn on $(\mP^2, Q)$ and we set:
$$e_5 \leftrightarrow_{Q} [\epsilon_5].
$$
We denote by $L_5$ the unique line of $\mP^2$ such that $e_5=e_{5}(Q):={\rm{Pol}}_{Q}(L_{5})$. We have seen that $\beta_{14}=\beta_{15}$. We know that 
 $\epsilon_5\cap\beta_{14}\neq\emptyset$. By the identification $\leftrightarrow_{Q}$ it holds that $n_{14}\in L_5$. In other words $L_5$ belongs to the same pencil given by $L_1$ and $L_4$. 

Now we consider $g\colon\mP^2\to\mP^2$ a projective isomorphism such that 
$$
g(L_i)=L_i,\,\, i=1,2,3,{\rm{and}} \,\, g(L_4)=L_5, g(L_5)=L_4.
$$

Then it holds that 
$$
g(n_{12})= n_{12}, g(n_{13})=n_{13}, g(n_{23})=n_{23},\,\, {\rm{and}}\, \, g(n_{14})=n_{14},
$$
where we stress that the three points $n_{12},n_{13},n_{14}$ belongs to the line $L_1$.
We note that $g\colon\mP^2 \to \mP^2$ is not necessarily a $Q$-orthogonal isomorphism, that is if $a\perp_Q b$ then it is not true that $g(a)\perp_Q g(b)$. However it restricts to the identity on the line $L_1$ and it fix the other two lines $L_2$ and $L_3$, hence it maintains the orthogonality conditions concerning $e_1, e_2, e_3$, $L_1,L_2,L_3$ and $n_{12}, n_{13}, n_{23}, n_{14}$ we had before.

Let us consider the curve $g(M)$. Obviously $g(M)$ is isomorphic to $M$. By construction the configuration of lines associated to $g(M)$ is $\{ L_1,L_2,L_3,L_5 \}$. 

We have shown that for any smooth quadric $Q\in \mP_{14\perp 23}$ the configuration of lines associated to $g(M)$ gives back, by the identification $(\mP^2, Q) \leftrightarrow_{Q} ( \sH^{B}_{1},\Omega)$, four lines $\epsilon_1$, $\epsilon_2$, $\epsilon_3$, $\epsilon_5$ of $B$ which are disjoint, but they belongs to the same hyperplane $H$ of $\mP^6$. By Lemma \ref{smoothnessofZ} we have that $Z=B\cap H$ is a smooth del Pezzo surface. If $\epsilon_4\subset H$ the above discussion works by taking $g$ to be the identity. By our moduli problem we can from now on identify $M$ to $g(M)$.

To sum up we have shown for now only that given a general point $[C,\theta]\in {\rm{S}}^{{\rm{null}},0}_4$, we can find an automorphism $g\colon \mP^2\to\mP^2$ such that the configuration of lines $\{ L_1,L_2,L_3,L_4 \}$ associated to $g\circ \phi_{\theta+\delta}(C)=g(M)$ gives four points $e_1, e_2, e_3, e_4$ such that for any smooth $Q\in\mP_{14\perp 23}$ it holds that $e_{j}(Q):={\rm{Pol}}_{Q}(L_{j})$,  $j=1,2,3,4$ and by the identification $(\mP^2, Q) \leftrightarrow_{Q} ( \sH^{B}_{1},\Omega)$ the four lines, $\epsilon_1,\epsilon_2,\epsilon_3,\epsilon_4$, where $e_i \leftrightarrow_{Q} [\epsilon_i]$ are disjoint and span a hyperplane section $H$.

Now for a smooth $Q\in \mP_{14\perp23}$ we set:
$$M_{ij}:={\rm{Pol}}_{n_{ij}}(Q),\,\, (i,j)\in \mathfrak T.$$
\noindent By the identification $(\mP^2, Q) \leftrightarrow_{Q} ( \sH^{B}_{1},\Omega)$ it immediately follows that $M_{ij}$ is the line $\langle [\epsilon_{i}],[\epsilon_{j}] \rangle\subset \sH^B_1$, for every $(i,j)\in \mathfrak T$. Furthermore we have $\beta_{12}\cap\beta_{34}\neq\emptyset$. Then $n_{34}\in M_{12}$. We recall that the plane sextic $M$ has a simple node supported on $n_{34}$.

By generality of $[C,\theta]$ the line $M_{12}$ cuts $4$ distinct  points $x^{1}_{12},x^{2}_{12},y^{1}_{12},y^{2}_{12}$ on $M$, none of which coincides with $n_{34}$. Set $x^{1}_{34},x^{2}_{34},y^{1}_{34},y^{2}_{34}$ for the analogue ones with respect to $M_{34}$. We denote by $\alpha^{i}_{jk}, \beta^{i}_{jk}\subset B$ the corresponding lines; that is  $[\alpha^{i}_{jk}] \leftrightarrow_{Q} x^{i}_{jk}$ and respectively  $[\beta^{i}_{jk}] \leftrightarrow_{Q} y^{i}_{jk}$. We use the four degrees of freedom on $Q$ to impose the following four conditions:

\begin{equation}\label{forcebis}
x^{1}_{12}\perp_{Q} x^{2}_{12},\,\,y^{1}_{12}\perp_{Q} y^{2}_{12},\,\, x^{1}_{34}\perp_{Q} x^{2}_{34},\,\, y^{1}_{34}\perp_{Q} y^{2}_{34}.
\end{equation}
In terms of the geometry on $B$ the condition (\ref{forcebis}) means:
\begin{equation}\label{force}\alpha^{1}_{12}\cap \alpha^{2}_{12}\neq\emptyset,\,\,\beta^{1}_{12}\cap\beta^{2}_{12}\neq\emptyset,\,\, \alpha^{1}_{34}\cap\alpha^{2}_{34}\neq\emptyset,\,\, 
\beta^{1}_{34}\cap\beta^{2}_{34}\neq\emptyset.
\end{equation}

Consider now the section $Z=H\cap B$ that we have built above, by Lemma \ref{smoothnessofZ} $Z$ is smooth. For the polarity underlying the geometry of the lines of $B$ we see that the lines $\alpha^{1}_{12}, \alpha^{2}_{12}$ are not contained in $H$ but there exists a point $p_{12}\in \beta_{12}$ such that it belongs to both of them. The same holds for the triples $\{\beta^{1}_{12}, \beta^{2}_{12}, \beta_{12}\}$, $\{\alpha^{1}_{34}, \alpha^{2}_{34}, \beta_{34}\}$, $\{\beta^{1}_{34}, \beta^{2}_{34}, \beta_{34}\}$. To sum up  there exist four
 points $p_{12},q_{12} ,p_{34},q_{34}\in Z$ such that $p_{12},q_{12}\in\beta_{12}\subset Z\subset B$, $p_{34},q_{34}\in\beta_{34}\subset Z\subset B$ such that it holds the following:
  \begin{equation}\label{forcetris}\alpha^{1}_{12}\cap \alpha^{2}_{12}=\{p_{12}\},\,\,
\beta^{1}_{12}\cap\beta^{2}_{12}=\{q_{12}\},\,\, \alpha^{1}_{34}\cap\alpha^{2}_{34}=\{p_{34}\},\,\,
\beta^{1}_{34}\cap\beta^{2}_{34}=\{q_{34}\}.
\end{equation}

  Now we can use the notations given in \ref{useesu} for the hyperplane section $Z$. By generality assumptions the sublinear system $\Lambda(p_{12}+q_{12}+ p_{34}+q_{34})\subset \mid 2\lambda\mid$ given by those divisors passing through the points $p_{12},q_{12} ,p_{34},q_{34}\in Z$ is a pencil. To each general element $R'\in \Lambda(p_{12}+q_{12}+ p_{34}+q_{34})$ we associate the corresponding $C(R')$ as in Lemma \ref{primaCIII}. We want to find the unique $R\in \Lambda(p_{12}+q_{12}+ p_{34}+q_{34})$ such that $C=C(R)$. 

To ease reading we stress that on one side we have the blow-up at four points $\phi_{|\lambda|}\colon Z\to\mP^2$,  which gives the hyperplane section $Z=B\cap H$ containing the rational sextics $R'\in \Lambda(p_{12}+q_{12}+ p_{34}+q_{34})$, and on the other side we have the blow-up $\sigma\colon S\to\mP^2$ at the six points $n_{ij}$, $(i,j)\in \mathfrak T$ which contains the given genus $4$ curve $C$. Clearly by the identification  $(\mP^2, Q)\simeq( \sH^{B}_{1},\Omega)$ we can see $C(R')$ inside $S$ for every smooth $R'\in \Lambda(p_{12}+q_{12}+ p_{34}+q_{34})$.

We work on $S$. Denote by $L\in  |h-E_{34}|$ and $L'\in  |h-E_{12}|$ the $\sigma$-strict transforms of $M_{12}$ and respectively of $M_{34}$.
 By construction it holds that the points $a^t_{nm}, b^{i}_{jk}$ belong to $C\cap C(R)$ where $\sigma(a^t_{nm})=x^t_{nm}$, $t=1,2$, $(n,m)=(1,2)$ or $(3,4)$ and  $\sigma (b^{i}_{jk})=y^i_{jk}$, $i=1,2$, $(j,k)=(1,2)$ or $(3,4)$. Note that for any irreducible element $C'\in |6h-2(E_{12}+ E_{13}  + E_{14} + E_{23} + E_{24} +E_{34})|$ it holds that ${\rm{deg}}C'_{|L}=4$ and ${\rm{deg}}C'_{|L'}=4$. In particular the decomposition sequence gives $h^{0}(L\cup L',\sO_{L\cup L'}(C'))=9$. By Kodaira vanishing and by Riemann Roch theorem it holds that $h^{0}(S,\sO_{S}(C'))=10$. Moreover $|4h-2( E_{13}  + E_{14} + E_{23} + E_{24}) -E_{34} -E_{12}|$ contains the unique element $D=(h-E_{12}-E_{13} - E_{14})+(h-E_{34}-E_{24} - E_{14})+(h-E_{12}-E_{23} - E_{24})+E_{12}+(h-E_{13}-E_{23} - E_{34})+E_{34}$. By the cohomology of the standard sequence
$$
0\to\sO_S(  -L-L' )\to \sO_S\to\sO_{L\cup L'}\to 0
$$tensorialised by $\sO_{S}(C')$ it follows that the following sequence is then exact:
\begin{equation}\label{globalsection}
0\to H^{0}(S, \sO_S(D)) \to H^{0}(S, \sO_{S}(C'))\stackrel {{\rm{ev}}}{\longrightarrow} H^{0}(L\cup L',\sO_{L\cup L'}(C'))\to 0.
\end{equation}
We know that $h^{0}(S, \sO_S(D))=1$. Now let $\langle \nu\rangle\subset H^{0}(L\cup L',\sO_{L\cup L'}(C'))$ be the $1$-dimensional vector space given by those sections $\nu$ vanishing on $a^{1}_{12}+a^{2}_{12}+b^{1}_{12}+b^{2}_{12}+a^{1}_{34}+a^{2}_{34}+b^{1}_{34}+b^{2}_{34}$. Then ${{\rm{ev}}}^{-1}(\langle \nu\rangle)$ is a $2$-dimensional vector space $\Lambda$.
Then for every $R'\in \Lambda(p_{12}+q_{12}+ p_{34}+q_{34})$ the corresponding curve $C(R')$ belongs to the pencil $\mP(\Lambda)$ given by those $C'$ such that 
$C'_{|L\cup L'}= a^{1}_{12}+a^{2}_{12}+b^{1}_{12}+b^{2}_{12}+a^{1}_{34}+a^{2}_{34}+b^{1}_{34}+b^{2}_{34}$. By construction $C$ belongs to $\mP(\Lambda)$ too. Finally consider a general point $a\in C$. Thanks to the morphism $\sigma\colon S\to\sH^B_1$ the point $\sigma(a)=[\alpha]$ gives a line $\alpha\subset B$ not contained in $Z$. Consider the unique point $p\in\alpha\cap Z$. Let $R$ be the unique element $R\in \Lambda(p_{12}+q_{12}+ p_{34}+q_{34})$ such that $p\in R$. Then $[\alpha]\in M\cap M(R)$; that is $a\in C\cap C(R)$. This implies that $C=C(R)$. 
  \end{proof}
In the proof of Theorem \ref{perlomodulothree} we have used the following Lemma.

\begin{lem}\label{smoothnessofZ} Let $Z=B\cap H$ be an hyperplane section of the del Pezzo $3$-fold $B$ such that it contains $10$ lines $e_1,e_2, e_3, e_4$ and  $\beta_{ij}$ $(i,j)\in \mathfrak T$ such that $e_i\cap \beta_{ij}\neq \emptyset$ and $e_j\cap \beta_{ij}\neq \emptyset$ then $Z$ is smooth.
\end{lem}

\begin{proof} By \cite[Lemma 7.6.2.Remark 7.6.3]{CS} we know that if $Z$ has at most isolated singularities then $Z$ cannot contain more than $7$ lines. If $Z$ has a non-normal singularity then it is singular only along a line since ${\rm{Pic}}(B)=\mathbb Z$ and \cite[Lemma page 718]{R}; then the claim follows.
\end{proof}

\section{The rationality theorem of the theta-null spin$^{+}$ moduli space}

\subsection{The moduli map}\label{Sub: Two ray} In the Introduction we have recalled the rational map 
$\pi_{{\rm{S}}^{+}_{4}}\colon\sH^{B}_{6}\dashrightarrow {\rm{S}}^{+}_{4}$ and its Stein factorisation $q_{{\rm{S}}^{+}_{4}}\circ
p_{{\rm{S}}^{+}_{4}}$ where the general fiber of $p_{{\rm{S}}^{+}_{4}}\colon \sH^{B}_{6}\dashrightarrow
\widetilde{{\rm{S}}}^{+}_{4}$ is irreducible and $q_{{\rm{S}}^{+}_{4}}\colon \widetilde{{\rm{S}}}^{+}_{4}\dashrightarrow
{{\rm{S}}}^{+}_{4}$ is generically finite of degree $2$; see: \cite[Corollary 4.16]{TZ3}. 
Our theory on $\pi_{{\rm{S}}^{+}_{4}}\colon\sH^{B}_{6}\dashrightarrow {\rm{S}}^{+}_{4}$ can be straightly extended to the divisor $\sH_{{\rm{c.t.}}}$.
Actually the proof that ${\rm{S}}^{+}_{4}$ is a rational variety was done by taking the open subset ${\sH}^*$ of $\sH^{B}_{6}$
consisting of (reduced but possibly reducible)
sextic curves with exactly six different bi-secant lines. On it in \cite[Susection 5.1]{TZ3} we defined a $G$-equivariant morphism 
\begin{equation*}
\label{eq:Theta}
\Theta\colon {\sH}^* \rightarrow (\mP^2)^6/\mathfrak{S}_6,\,\,
[R]\mapsto ([\beta_{1}],\ldots,[\beta_{6}]).
\end{equation*}
which associates to each $[R]\in {\sH}^*$ the set of its $6$ unordered bisecants, where $(\mP^2)^6$ is the product of $6$ copies of $\mP^2$, $\mathfrak{S}_6$ is the permutations group on $6$ elements and $G={\rm{PGL}}(2\mathbb C)$. In \cite[Theorem 5.2]{TZ3} we showed that $\Theta$ is an isomorphism on a certain open set $\mathring \sH$, see:\cite[Condition 3.21]{TZ3}, which is disjoint from $\sH_{\rm{{c.t.}}}$. On the other hand in \cite[Condition 6.9]{TZ3} we defined a locally closed subset $\sD$ of $\sH^{B}_{6}$, which contains $\sH_{\rm{{c.t.}}}$, and we could extend $\Theta$ to an isomorphism over $\mathring\sH\cup \sD$. In particular $\Theta$ is an isomorphism over $\mathring \sH_{\rm{{c.t.}}}\hookrightarrow \sD$, where we recall here that $\mathring \sH_{\rm{{c.t.}}}\hookrightarrow\sH_{\rm{{c.t.}}}$ is the open subscheme of $\sH_{\rm{{c.t.}}}$ given by those $[R]\in\sH_{\rm{{c.t.}}}$ such that $R$ is smooth and the hyperplane section $\langle R\rangle\cap B=Z$ is smooth.

To ease reading we follow the notation of {\cite{TZ3} and we set
\[
\mathring{\widetilde{\rm{S}}_4^{+}}\subset \widetilde{\rm S}^{+}_{4}
\]
for the image of 
$\mathring\sH$ 
and we denote by
\[
\mathring{{\rm S}_4^{+}} \subset \rm S^+_4
\]
 the image of $\mathring{\widetilde{\rm S}_4^{+}}$ by the map $q_{{\rm{S}}^{+}_{4}}\colon \widetilde{{\rm{S}}}^{+}_{4}\dashrightarrow
{{\rm{S}}}^{+}_{4}$.
By \cite[Corollary 4.16]{TZ3} 
we know that there exists an involution on 
$\tilde{\mathring J}\colon \widetilde{\rm S}_4^{+o}\to \widetilde{\rm S}_4^{+o} $,
which is the deck transformation of the double cover
$q_{{\rm{S}}^{+o}_{4}}:=q_{{\rm{S}}^{+}_{4_{\mid\widetilde{\rm S}_4^{+o} }}}\colon \widetilde{\rm S}_4^{+o}\to \rm S^{+o}_4$. Moreover letting $V_1:= \Theta(\mathring \sH)$ by  \cite[Proposition 6.1]{TZ3} we know that there exists a commutative diagram

\begin{equation}
\xymatrix{ V_1/G 
\ar[r]^{ \mathring J}\ar[d]_{[\Theta^{-1}]}& V_1/G
\ar[d]_{[\Theta^{-1}]}&  \\
\widetilde{{\rm{S}}}^{+o}_{4} \ar[r]^{\tilde {\mathring J}}\ar[d]_{q_{{\rm{S}}^{+o}_{4}}}& \widetilde{{\rm{S}}}^{+o}_{4}\ar@{-->}[d]_{q_{{\rm{S}}^{+o}_{4}}}\\
{{\rm{S}}}^{+0}_{4}\ar[r]^{\rm{id}} & {{\rm{S}}}^{+o}_{4}}
\end{equation}
where $\mathring J\to V_1/G\to V_1/G$ is a lifting of the classical association map 
$$j\colon (\mP^2)^{6}/\!/\PGL_3/\mathfrak{S}_6\to (\mP^2)^{6}/\!/\PGL_3/\mathfrak{S}_6$$ see: \cite[Subsection 6.1]{TZ3}. Now set $\mathring \sK:=\Theta(\mathring \sH_{\rm{{c.t.}}})$

\begin{lem}\label{fissati} We can extend $J\colon V_1/G\to V_1/G$ to $(V_1\cup \mathring \sK)/G$ in such a way that it is the identity over $\mathring \sK/G$.
\end{lem}
\begin{proof} We know that $j\colon (\mP^2)^{6}/\!/\PGL_3/\mathfrak{S}_6\to (\mP^2)^{6}/\!/\PGL_3/\mathfrak{S}_6$  is the identity on the points given by the sextuplets of points of $\mP^2$which are contained in a conic and that the $\PGL_3$-action and the $\mathfrak{S}_6$-action over $(\mP^2)^{6}$ commutes. On the other hand if $[R]\in \mathring \sH_{\rm{{c.t.}}}$ then its $6$ bisecant are contained in a (reducible) conic. More precisely, following the notation of \ref{useesu} it holds that the bisecants $\beta_{ij}$, $1\leq i<j\leq 4$, belongs to the reducible conic which parameterises the lines of $B_5$ which touch $\epsilon_1$ or $\epsilon_4$. Note that by \cite[Theorem 1, p.23]{DO}) these points are stable ones and by \cite[p.118--120]{DO}) they are fixed by $j$.
\end{proof}

We consider the morphism $[\Theta^{-1}]\colon (V_1\cup \mathring \sK)/G\to \mathring(\sH\cup \mathring \sH_{\rm{{c.t.}}}) /\!/G$ induced by $\Theta$. We stress that by definition $[\Theta^{-1}](\mathring \sK/G)=\mathring\sH_{\rm{{c.t.}}} /\!/G=p_{{\rm{S}}^{+}_{4}}(\mathring\sH_{\rm{{c.t.}}})$.

\begin{prop}\label{estensionee} We can extend $q_{{\rm{S}}^{+o}_{4}}\circ [\Theta^{-1}]_{|V_1/G}\colon  V_1/G\to\mathring{\rm S^{+}_4}$ to 
$f=q_{{\rm{S}}^{+o}_{4}}\circ [\Theta^{-1}]\colon (V_1\cup \mathring \sK)/G\to \mathring{\rm S^{+}_4}\sqcup {{{\rm{S}}^{{\rm{null}},0 }_{g}}}$ in such a way that $f_{|\mathring \sK/G}\colon \mathring \sK/G\to {{{\rm{S}}^{{\rm{null}},0 }_{g}}}$ is dominant of degree $1$. \end{prop}
\begin{proof}  By \cite[Proposition 6.10]{TZ3} we have that $\Theta$ is extendable over $\mathring \sH_{\rm{{c.t.}}}$ so as to give an isomorphism over it. By Proposition \ref{newIII} $\pi_{{\rm{S}}^{+}_{4}}\colon\sH^{B}_{6}\dashrightarrow {\rm{S}}^{+}_{4}$ is defined over $\mathring \sH_{\rm{{c.t.}}}$, that is it is defined ${{{\rm{S}}^{{\rm{null}},0 }_{g}}}\ni [C(R),\theta(R)]:=q_{{\rm{S}}^{+}_{4}}\circ
p_{{\rm{S}}^{+}_{4}}([R])$ where $[R]\in \mathring\sH_{\rm{{c.t.}}}$. This implies that if $[(\beta_1,\cdots \beta_6)]\in \mathring \sK/G$ and $[R]\in \mathring\sH_{\rm{{c.t.}}}$ is such that $[R]=\Theta^{-1}((\beta_1,\cdots \beta_6))$ then it is defined $f[(\beta_1,\cdots \beta_6)]\mapsto [C(R),\theta(R)]$ where $f$ coincides with $ q_{{\rm{S}}^{+o}_{4}}\circ[\Theta^{-1}]$ over $V_1/G$. Finally we have realised $\mathring{\rm S^{+}_4}$ as the $\mathbb Z/2\mathbb Z$-quotient $(V_1/G)/\mathring J$ where $f_{|V_1/G}=\Theta^{-1}\circ q_{{\rm{S}}^{+o}_{4}}\colon  V_1/G\to\mathring {\rm S^{+}_4}$ is the associated involution. Since $f$ is defined on  $\mathring \sK/G$ and since by Lemma \ref{fissati} $J$ is the identity on $\mathring \sK/G$ it follows that $f(\mathring \sK/G)\subset {{{\rm{S}}^{{\rm{null}},0 }_{g}}}$ is in the branch loci of $f$. On the other hand by the Reconstruction theorem \ref{perlomodulothree} we know that $\pi_{{\rm{S}}^{+}_{4}}\colon \mathring \sH_{\rm{{c.t.}}}\to {{{\rm{S}}^{{\rm{null}},0 }_{g}}}$ is dominant, hence $f_{|\mathring \sK/G}\colon \mathring \sK/G\to {{{\rm{S}}^{{\rm{null}},0 }_{g}}}$ is dominant of degree $1$. \end{proof}

\subsection{The proof of the rationality theorem}
We are now ready to show our main theorem. First we sum up part of the result of the above Section \ref{Sub: Two ray} into the following proposition:

\begin{prop}\label{forrationalityvanishing} The divisor ${{{\rm{S}}^{{\rm{null}},0}_{4}}}$ is irreducible reduced and it is dominated by $\mathring{\sH}_{{\rm{c.t.}}}$. Moreover it is contained in the branch loci of $q_{{\rm{S}}^{+}_{4}}\colon \widetilde{{\rm{S}}}^{+}_{4}\dashrightarrow{{\rm{S}}}^{+}_{4}$.In particular it holds that $p_{{\rm{S}}^{+}_{4}}(\sH_{\rm{{c.t.}}})$ is birational to ${{{\rm{S}}^{{\rm{null}},0 }_{g}}}$.

\end{prop}
\begin{proof} By Proposition \ref{reducedness}  ${{{\rm{S}}^{{\rm{null}},0}_{4}}}$ is reduced. Moreover we have seen that by definition $[\Theta^{-1}](\mathring \sK/G)=\mathring\sH_{\rm{{c.t.}}} /\!/G=p_{{\rm{S}}^{+}_{4}}(\mathring\sH_{\rm{{c.t.}}})$. By Proposition \ref{cleaning} we know that $\mathring{\sH}_{{\rm{c.t.}}}$ is irreducible. Hence $\mathring \sK/G$ is irreducible.
By Proposition \ref{estensionee} it holds that that ${{{\rm{S}}^{{\rm{null}},0}_{4}}}$ is irreducible too.  By Proposition \ref{estensionee} it follows that the divisor ${{{\rm{S}}^{{\rm{null}},0}_{4}}}$ is contained in the branch loci of $q_{{\rm{S}}^{+}_{4}}\colon \widetilde{{\rm{S}}}^{+}_{4}\dashrightarrow{{\rm{S}}}^{+}_{4}$ and that it is birational to $p_{{\rm{S}}^{+}_{4}}(\mathring\sH_{\rm{{c.t.}}})$.
\end{proof}
 
 \begin{thm}\label{rationalityvanishing}
     ${\overline{{\rm{S}}^{{\rm{null}},0 }_{4}}}$ is a rational variety.     
 \end{thm}
 \begin{proof} 
By Proposition \ref{forrationalityvanishing} ${\overline{{\rm{S}}^{{\rm{null}},0 }_{4}}}$ is birational to  $p_{{\rm{S}}^{+}_{4}}(\mathring{\sH_{\rm{c.t.}}})$. On the other hand by construction  $p_{{\rm{S}}^{+}_{4}}(\mathring{\sH_{\rm{c.t.}}})$ is birational to $\mathring{\sH_{\rm{c.t.}}}//G$.
  By Theorem \ref{firstrational} we conclude.
 \end{proof}
\subsection{The rationality theorem of the theta-null Prym moduli space}
Let $\sR^{{\rm{null}}}_4$ be the moduli space which parameterises couples $(C,\eta)$ where
$C$ is a genus $4$ curve with a vanishing theta-null and $\eta$ is a nontrivial $2$-torsions line bundle. We denote by $\delta$ the vanishing theta-null, i.e. the $g^1_3$.
\begin{thm}\label{rationalityprym}
$\sR^{{\rm{null}}}_4$ is a rational variety.
\end{thm}
\begin{proof}
Let $f\colon {\rm{S}}^{{\rm{null}},0}_4\dashrightarrow \sR^{{\rm{null}}}_4$ be the rational map given by $[(C,\theta)]\mapsto [(C, \theta-\delta)]$. It is
generically injective. By a dimensional count this implies that ${\rm{S}}^{{\rm{null}},0}_4$ is birational to an irreducible component of  $\sR^{{\rm{null}}}_4$. Now, using our reconstruction theorem, that is Theorem \ref{perlomodulothree}, we show that  $\sR^{{\rm{null}}}_4$ is irreducible.

Indeed let  $[(C, \eta)]\in \sR^{{\rm{null}}}_4$ be a general element. Clearly our claim is to show that $h^0(C,\sO_C(\eta+\delta))=0$. By the exact sequence
$$0\to \sO_C(\eta)\to \sO_C(\eta+\delta)\to \sO(\eta+\delta)|_{\delta}\to 0$$
and the fact that $\sO(\eta+\delta)|_{\delta}\simeq \mC^{\oplus 3}$,
we see that $h^0(C,\sO_C(\eta+\delta))=0$ iff $H^0(\sO(\eta+\delta)|_{\delta})\simeq \mC^3\to H^1(\sO(\eta))\simeq \mC^3$ is an isomorphism.  We set $\theta:= \eta+\delta$. We consider the morphism $\phi_{|\theta+\delta|}\colon C\to\mP^2$. By the same argument used in Proposition \ref{threesexstic} the image $M$ of the morphism $\phi_{|\theta+\delta|}\colon C\to\mP^{2}$ is a sextic with six nodes. Let $\pi\colon S\to \mP^2$ be the blow-up at the six points
and $e_1,\dots, e_6$ the respective exceptional divisors.
The main point is again to show $|\pi^*\sO(4)-2\sum e_i|$ has a unique member.
Since $K_C=3h-\sum e_i|_C$ and $2h=2(\delta+\theta)=2K_C$,
we have $4h-2\sum e_i|_C=4h-2(3h-K_C)=2K_C-2h=0$.
Now consider the following exact seqeunce:
\[
0\to \pi^*\sO_{\mP^2}(-2)\to \pi^*\sO_{\mP^2}(4)\otimes_{\sO_S}\sO_S(-2\sum e_i)\to 
\sO_C(4h-2\sum e_i)\to 0
\]
(note that $C\in |\pi^*\sO_{\mP^2}(6)\otimes_{\sO_S}\sO_S(-2\sum e_i)|$).
By this, analogously as in the proof of Proposition \ref{threesexstic}, we see that $|\pi^*(\sO_{\mP^2}(4)-2\sum e_i|$ has a unique member and again the member $D$ of $|\pi^*(\sO(4)-2\sum e_i|$ consists of $4$ lines and  again the six points are all the mutual intersection points among them. Finally, by Theorem \ref{perlomodulothree} we can conclude that $\mathring{\sH}$ dominates also $\sR^{{\rm{null}}}_4$. Hence $\sR^{{\rm{null}}}_4$ is irreducible. In particular it is birational to ${\overline{\sS^{{\rm{null}},0 }_{4}}}$.  By Corollary \ref{rationalityvanishing} it is a rational variety.

\end{proof}

\end{document}